\documentclass[12pt]{article}
\pdfoutput=1

\usepackage{amssymb,amsmath,bbm}
\usepackage{amsthm}
\usepackage{graphicx}
\usepackage{fullpage}
\usepackage{cite}
\usepackage{hyperref}

\usepackage{amssymb,amsmath}

\renewcommand{\a}{\alpha}
\renewcommand{\b}{\beta}
\newcommand{\g}{\gamma}

\newcommand{\e}{\epsilon}
\newcommand{\z}{\zeta}
\newcommand{\h}{\eta}
\renewcommand{\th}{\theta}\newcommand{\Th}{\Theta}

\renewcommand{\l}{\lambda}

\newcommand{\p}{\pi}

\newcommand{\s}{\sigma}\renewcommand{\S}{\Sigma}
\renewcommand{\t}{\tau}

\newcommand{\ch}{\chi}

\renewcommand{\O}{\Omega}

\newcommand{\cC}{\mathcal{C}}

\newcommand{\cO}{\mathcal{O}}

\newcommand{\IC}{\mathbb{C}}

\newcommand{\IP}{\mathbb{P}}

\newcommand{\IR}{\mathbb{R}}

\newcommand{\IZ}{\mathbb{Z}}

\def\str{\vrule height14pt width0pt depth8pt}

\def\place#1#2#3{\vbox to0pt{\kern-\parskip\kern-7pt
                             \kern-#2truein\hbox{\kern#1truein #3}
                             \vss}\nointerlineskip}

\newcommand{\cicy}[2]{\begin{matrix} #1\end{matrix}\!\left[\begin{matrix}#2 \end{matrix}\right]}

\newcommand{\ii}{\text{i}}
\newcommand{\into}{\hookrightarrow}
\newcommand{\mb}[1]{\mathbf{#1}}
\newcommand{\quotient}[1]{_{\hskip-2pt\lower1pt\hbox{$/$}\lower2pt\hbox{\hskip-1pt$#1$}}}

\newcommand{\spn}[1]{\langle #1 \rangle}
\newcommand{\Xh}{\widehat X}
\newcommand{\Xt}{\widetilde X}
\newcommand{\xt}{\tilde x}
\newcommand{\Yt}{\widetilde Y}

\newcommand{\sref}[1]{section~\ref{#1}}

\newcommand{\fref}[1]{Figure~\ref{#1}}

\newcommand{\eref}[1]{equation~\eqref{#1}}

\newcommand{\Dic}{\mathrm{Dic}}
\newcommand{\emptysetnew}{\mathrm{O}\hspace{-7.6pt}\raisebox{1pt}{/}\,}

\newcommand{\Tr}{\mathrm{Tr}}

\newcommand{\hodgenos}{(h^{1,1},\,h^{2,1})}
\newcommand{\symm}[1]{_{\hskip-3pt\lower3pt\hbox{$\left\{#1\right\}$}}}

\newcommand{\cicystop}{~\lower8pt\hbox{.}}

\def\place#1#2#3{\vbox to0pt{\kern-\parskip\kern-7pt
                             \kern-#2truein\hbox{\kern#1truein #3}
                             \vss}\nointerlineskip}

\newcommand{\beq}{\begin{equation}}
\newcommand{\eeq}{\end{equation}}
\newcommand{\bea}{\begin{eqnarray}}
\newcommand{\eea}{\end{eqnarray}}
\newcommand{\bean}{\begin{eqnarray*}}
\newcommand{\eean}{\end{eqnarray*}}

\newcommand{\comment}[1]{}

\newcommand{\id}{{\bf 1}}

\newcommand{\pt}{\text{pt}}

\newtheorem{theorem}{Theorem}[section]
\newtheorem{lemma}[theorem]{Lemma}
\newtheorem{corollary}[theorem]{Corollary}
\theoremstyle{definition}
\newtheorem*{definition}{Definition}
\newtheorem{example}{Example}[section]
\theoremstyle{remark}
\newtheorem*{remark}{Remark}

\begin{document}
\thispagestyle{empty}

\renewcommand{\thefootnote}{$\dagger$}
\begin{center}
{\LARGE Classification and Properties of\\ Hyperconifold Singularities and Transitions \\}
\vspace{.5in}
Rhys Davies\footnote{\it daviesr@maths.ox.ac.uk}
\\
\vspace{.15in}
{\it
Mathematical Institute\\
University of Oxford\\
Andrew Wiles Building\\
Radcliffe Observatory Quarter\\
Woodstock Road, Oxford\\
OX2 6GG, UK}
\end{center}

\renewcommand{\thefootnote}{\arabic{footnote}}
\setcounter{footnote}{0}

\begin{abstract}
This paper is a detailed study of a class of isolated Gorenstein threefold\linebreak
singularities, called hyperconifolds, that are finite quotients of the conifold.
First, it is shown that hyperconifold singularities arise naturally in limits of
smooth, compact Calabi--Yau threefolds (in particular), when the group
action on the covering space develops a fixed point.  The
$\IZ_n$-hyperconifolds---those for which the quotient group is cyclic---are
classified, demonstrating a one-to-one correspondence
between these\linebreak singularities and three-dimensional lens spaces $L(n,k)$,
which occur as the vanishing\linebreak cycles.  The classification is constructive, and
leads to a simple proof that a \linebreak$\IZ_n$-hyperconifold is mirror to an $n$-nodal
variety.  It is then argued that all factorial $\IZ_n$-hyperconifolds have
crepant, projective resolutions, and this gives rise to transitions between
smooth compact Calabi--Yau threefolds, which are mirror to certain
conifold transitions. Formulae are derived for the change in both
fundamental group and Hodge numbers under such hyperconifold
transitions.

Finally, a number of explicit examples are given, to illustrate how to construct
new Calabi--Yau manifolds using hyperconifold transitions, and also to
highlight the differences which can occur when these singularities occur
in non-factorial varieties.
\end{abstract}

\newpage
\setcounter{page}{0}
\thispagestyle{empty}
\tableofcontents

\newpage
\section{Introduction and motivation} \label{sec:intro}

One of the intriguing features of three-dimensional Calabi--Yau
manifolds\footnote{For present purposes, a Calabi--Yau $n$-fold is projective,
with trivial canonical class, and Hodge numbers $h^{0,i} = 0$ for
$i=1,\ldots,n-1$.} is that many (perhaps all) deformation families are connected
by topological transitions: a smooth space is deformed until it becomes mildly
singular, and the singularity is then resolved in such a way as to preserve the
Calabi--Yau conditions.  Unlike the two-dimensional case of $K3$ surfaces, the
resulting manifold is not usually diffeomorphic to the one with which we started.
The best-known of these transitions is the conifold transition, where the
intermediate singular variety is nodal \cite{Clemens:1983db,Hirzebruch:1985tc}.
The purpose of this paper is to describe another class of transitions, introduced
in \cite{Davies:2009ub,Davies:2011is}, called \emph{hyperconifold transitions}.
A hyperconifold singularity is a quotient of a node (a precise definition shall be
given shortly), and this paper focusses on the cases where the quotient group
is a finite cyclic group.

Hyperconifold singularities are of interest for various reasons.  They arise very
naturally in compact Calabi--Yau threefolds, and the associated transitions can
change the fundamental group, which cannot happen in a conifold transition.
The best way to understand these statements is to first discuss smooth
multiply-connected Calabi--Yau threefolds.

Odd-dimensional Calabi--Yau manifolds can have non-trivial fundamental
groups, and multiply-connected Calabi--Yau threefolds play an
important role in heterotic string model building
\cite{Candelas:1985en,Witten:1985bz}.  In practice, these manifolds can
typically be constructed by the following procedure:
\begin{enumerate}
    \item
    A simple algebraic variety $Z$ is given (a toric variety, say), along with
    equations which cut out a simply-connected Calabi--Yau threefold\footnote{Such
    manifolds usually have moduli.  Unless it is important, I will not distinguish
    between a Calabi--Yau manifold, a family of Calabi--Yau manifolds, and a
    generic member of such a family.} $\Xt$.
    \item
    A finite group $G$ is chosen, along with an action of this group by
    automorphisms on the ambient space $Z$.
    \item
    If we can simultaneously choose $\Xt$ to be smooth, invariant under $G$, and
    to miss the fixed points of the $G$ action, then $X := \Xt/G$ is again a
    Calabi--Yau manifold, with fundamental group $\pi_1(X) \cong G$.
\end{enumerate}
In the case where $Z$ is a product of projective spaces, and $\Xt$ a complete
intersection therein (a so-called `Complete Intersection Calabi--Yau', or CICY
\cite{Candelas:1987kf}),
such group actions have been completely classified \cite{Braun:2010vc},
and this gives rise to by far the largest known class of multiply-connected
Calabi--Yau threefolds.

Restricting to $G$-invariant $\Xt$ typically leaves some freedom in the defining
polynomials (corresponding to moduli of the quotient space $X$), and it is often
possible to deform them until $\Xt$ contains a single fixed point of the $G$
action.  When this occurs, $\Xt$ necessarily becomes singular:
\begin{theorem}
    Let $\Xt_t$ be a family of Calabi--Yau threefolds, with parameter $t$, admitting
    an action by a finite group $G$ that maps each fibre to itself.  Suppose that, for
    $t\neq 0$, $\Xt_t$ is smooth, and the $G$ action is fixed-point-free, but that
    $\Xt_0$ contains a single $G$-fixed point.  Then $\Xt_0$ is singular at this
    fixed point.
\end{theorem}
\begin{proof}
    Suppose to the contrary that $\Xt_0$ is smooth at the fixed point.  Then the
    quotient $X_0 = \Xt_0/G$ contains an isolated orbifold singularity, and by
    assumption, $X_t = \Xt_t/G$ gives a smoothing of this singularity.  But orbifolds
    of codimension $\geq 3$ are rigid \cite{Schlessinger:1971rg}, so this is a
    contradiction.
\end{proof}
\begin{remark}
    The same result holds if $\Xt_0$ contains an isolated fixed point of some
    non-trivial subgroup $H \leq G$.  In the case $H \neq G$, there will be
    multiple singularities, filling out a $G$-orbit.
\end{remark}
There is no indication from this argument about what \emph{kind} of singularity will
occur on $\Xt$, but we might expect that generically it will be the simplest
possible---an ordinary double point.  Indeed, this is what is almost always found
in practice. Before continuing with generalities, it is helpful to look at an explicit
example.

\begin{example}[The $\IZ_5$ quotient of the quintic]
    Perhaps the simplest family of Calabi--Yau threefolds is the quintic
    hypersurfaces in $\IP^4$, and a sub-family of these admits a free action by
    $\IZ_5$, which I will now describe.  If we take homogeneous coordinates $z_i$
    on $\IP^4$, where $i=0,\ldots,4$, then there is an action of $\IZ_5$ given by
    \begin{equation*}
        z_i \mapsto \z^i z_i ~~\mathrm{where}~~ \z = e^{2\p\ii/5} ~.
    \end{equation*}
    We can restrict our attention to homogeneous quintic polynomials $f$ which are
    invariant under this action, and the general such polynomial corresponds to a
    smooth Calabi--Yau hypersurface $\Xt$, on which $\IZ_5$ acts.\footnote{In
    examples, I will frequently state without proof that certain varieties are smooth,
    or that group actions are fixed-point-free.  These statements are all checked
    using the techniques described in \cite{Candelas:2008wb}.}
    
    Fixed points of the $\IZ_5$ action occur when exactly one of the homogeneous
    coordinates is non-zero.  Since each monomial $z_i^5$ is invariant, they can
    all occur as terms in $f$, so the generic invariant $\Xt$ misses the fixed points,
    and the quotient space $X = \Xt/\IZ_5$ is another smooth Calabi--Yau.  To see
    what happens when we allow $\Xt$ to intersect one of the fixed points, define
    affine coordinates $y_i = z_i/z_0$, $i=1,\ldots,4$; on this patch, the equation
    $f = 0$ becomes
    \begin{equation} \label{eq:quintic_expansion}
        0 = A_0 + A_1y_1y_4 + A_2 y_2 y_3 + \ldots ~,
    \end{equation}
    where the $A_m$ are arbitrary complex coefficients, and `$\ldots$' represents
    invariant higher-order terms.  The fixed point is given by $y_i = 0$ for all $i$,
    and the hypersurface will contain this point if and only if $A_0 = 0$.  In this
    case, equation \eqref{eq:quintic_expansion} begins with second-order terms,
    so that $\Xt$ is singular at the origin.  For generic choices of the other
    coefficients, the singularity will be a node, and it is easily checked that $\Xt$
    is still smooth elsewhere.  The quotient space $X$ therefore develops an
    isolated singularity that is a $\IZ_5$ quotient of a node.  This is our first
    example of a hyperconifold.
\end{example}

Let `the conifold' itself (the standard local model for a node, given as a
hypersurface in $\IC^4$ by $y_1 y_4 - y_2 y_3 = 0$) be denoted by $\cC$; we
will review its geometry in Section~\ref{sec:conifold}.    To define more carefully
the class of singularities we wish to study, we need to think about the
properties of group actions on $\cC$ which can arise as above.  First, note that
a free group action on a smooth compact Calabi--Yau threefold necessarily
leaves invariant the holomorphic\linebreak$(3,0)$-form $\O$ (this is a simple application
of the Atiyah-Bott fixed-point formula \cite{AtiyahBott,AtiyahBott2}), and therefore
acts trivially on the canonical bundle, which $\O$ generates.  Therefore the same
is true on the smooth locus of $\Xt_0$, and by continuity, the $G$-action on the
fibre over the singular point must be trivial.  The canonical sheaf of the quotient
is therefore also a bundle, i.e., the quotient space is Gorenstein.

The discussion so far motivates the following definition:
\begin{definition} \label{def:hyperconifold}
    Let $G$ be a finite group.  If $G$ acts by automorphisms on the conifold $\cC$,
    such that $\cC/G$ has an isolated Gorenstein singularity, then $\cC/G$ is called a
    \emph{$G$-hyperconifold}.\footnote{A note on terminology: The term `conifold'
    was introduced in \cite{Candelas:1989ug}, and was originally used to refer to a
    space which is singular only at a finite number of points, each modelled on a
    cone over some smooth space.  In the physics literature, however, `the conifold'
    is now widely used to mean the specific geometry $\cC$; nodal threefolds are
    sometimes said to `have conifold singularities'.  It is this second usage which is
    generalised by the term `hyperconifold'---a contraction of `hyperquotient' and
    `conifold'.}
\end{definition}
Although this allows for arbitrary finite groups $G$, the remainder of this paper
will focus on cyclic groups $\IZ_n$; other examples have not been studied.  One
of the main results will be a classification of $\IZ_n$-hyperconifolds (Theorem
\ref{th:classification}); with a single exception, they are toric
singularities, and this fact is very useful in their study.  Except in Appendix
\ref{app:exception}, where the exceptional case is discussed, we will restrict
our attention to the infinite class of toric hyperconifolds.

The goal of this work is to satisfactorily complete the study of hyperconifolds
begun in \cite{Davies:2009ub,Davies:2011is,Davies:2011fr}, by unifying,
generalising, and expanding upon the results of those papers.  The reader
might want to note that \cite{Davies:2009ub} is made obsolete by
the present paper, but the examples given in \cite{Davies:2011is,Davies:2011fr}
are not repeated here, and may still be of interest.  Furthermore, although this
project was initially motivated by applications to string theory, there is no
discussion of string theory herein, as I have nothing to add to the analysis of
\cite{Davies:2011is}, which described the physics of Type IIB string theory
compactified on a hyperconifold.

The outline of the paper is as follows.  In Section~\ref{sec:conifold}, I will review
the geometry of the conifold $\cC$, including its symmetry group.  Section
\ref{sec:classification} contains the classification of the\linebreak$\IZ_n$-hyperconifolds,
along with a description of their topology.  Their toric data is also given, and this
leads to a demonstration that the mirror of a $\IZ_n$-hyperconifold has
$n$ nodes.  In Section~\ref{sec:transitions}, it is proven that that any factorial
threefold $X_0$ with a hyperconifold singularity has a projective crepant
resolution, proving the ubiquity of hyperconifold \emph{transitions} between
compact Calabi--Yau manifolds.  Section~\ref{sec:examples} contains a
number of examples which illustrate the general results, including one which
is worked through in complete detail---I hope this will be useful as a reference
case.

\subsubsection*{Notation}

The notation used throughout shall be as follows.  $\cC$ is the conifold.  $\Xt$
denotes a smooth simply-connected Calabi--Yau threefold, and $X = \Xt/G$ its
quotient by a free action of the finite group $G$.  Sometimes the independent
Hodge numbers $h^{1,1}$ and $h^{2,1}$ will be appended as superscripts,
e.g., $X^{1,4}$ is a manifold with $\hodgenos = (1,4)$.  A singular deformation
of $\Xt$ will be denoted by $\Xt_0$, and its quotient by $X_0$.  The symbol
$\pi$ will be used for all resolution maps, and a hat will indicate the smooth
space corresponding to such a resolution, e.g., $\pi: \Xh \to X_0$.

\section{The conifold} \label{sec:conifold}

In order to establish conventions, I will first briefly review the pertinent features of
the conifold $\cC$ itself.  It is most simply defined as a hypersurface in $\IC^4$; if
we define complex `hypersurface coordinates' $(y_1, y_2, y_3, y_4)$, then $\cC$
is given by
\begin{equation} \label{eq:conifold}
    y_1 y_4 - y_2 y_3 = 0 ~.
\end{equation}
It is useful to have a symbol for the polynomial itself, so let
$p := y_1 y_4 - y_2 y_3$.

Being a Calabi--Yau, $\cC$ has trivial canonical class, and we can write down
an explicit expression for the nowhere-vanishing holomorphic $(3,0)$-form $\O$
on the smooth locus:
\begin{equation} \label{eq:Omega}
    \O = \oint_{p=0} \frac{dy_1\wedge dy_2 \wedge dy_3 \wedge dy_4}{p} ~.
\end{equation}

We can see that $\cC$ is a toric variety by making the following substitution,
which satisfies $p = 0$ identically:
\begin{equation} \label{eq:toric_coords}
    y_1 = \frac{t_1}{t_3} ~,~ y_2 = t_2 ~,~ y_3 = \frac{t_1}{t_2} ~,~ y_4 = t_3 ~,
\end{equation}
with $t_i \in \IC^*$.  Following standard procedures of toric geometry, we find that
$\cC$ corresponds to the rational polyhedral cone with vertices
\begin{equation*}
    v_1 = (1,0,0) ~,~~ v_2 = (1,1,0) ~,~~ v_3 = (1,0,1) ~,~~ v_4 = (1,1,1)~.
\end{equation*}
It is important to note that this choice of vertices is only unique up to the
action of $SL(3, \IZ)$, and a different choice would correspond to a different
relation between the $y$ and $t$\linebreak coordinates.  We can see that the primitive
lattice vectors which generate the cone lie in a common hyperplane.  The
same is true for any toric Calabi--Yau threefold, so it is customary to plot a
`toric diagram' showing just the intersection of the fan with this plane.  The
diagram for the conifold is shown in \fref{fig:conifold_fan}.
\begin{figure}[ht]
\begin{center}
    \includegraphics[width=.2\textwidth]{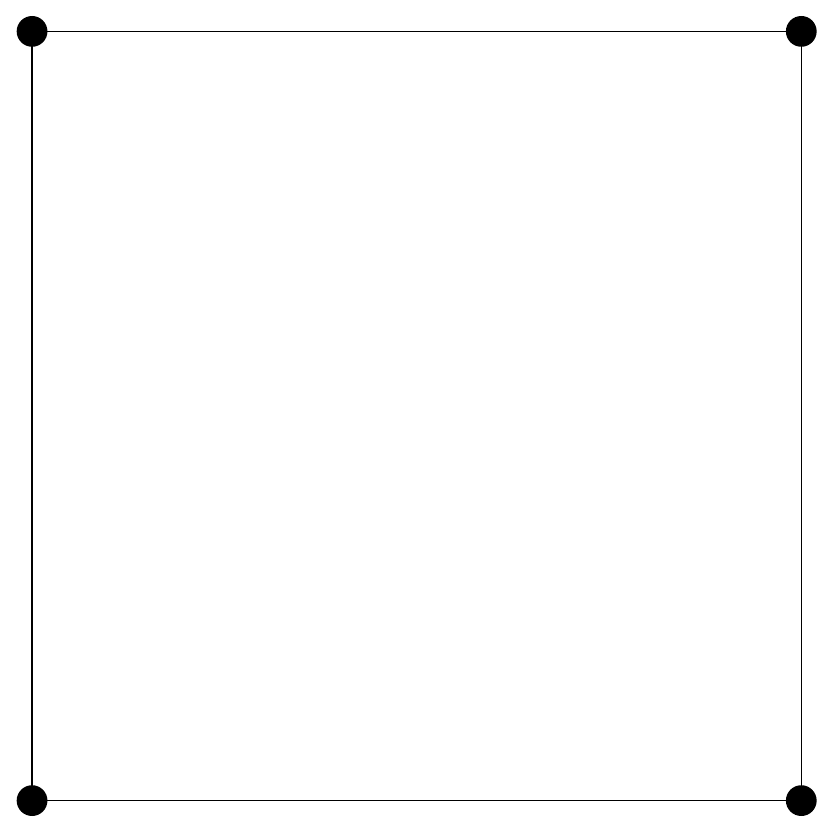}
    \caption{\small \label{fig:conifold_fan}
    The toric diagram for the conifold $\cC$.
    }
\end{center}
\end{figure}

It is also useful to introduce `homogeneous' coordinates for $\cC$, given by
the substitution
\begin{equation} \label{eq:homog_coords}
    y_1 = z_1 z_3 ,~ y_2 = z_1 z_4 ,~ y_3 = z_2 z_3 ,~ y_4 = z_2 z_4 ~,
\end{equation}
which also satisfies $p = 0$ identically.  These coordinates are subject to the
following\linebreak equivalence relation:
\begin{equation} \label{eq:rescaling}
    (z_1, z_2, z_3, z_4) \sim (\l z_1, \l z_2, \l^{-1} z_3, \l^{-1}z_4) ~,~ \l \in \IC^* ~.
\end{equation}
\subsection{Symmetries}

We have already established that $\cC$ is toric, and hence admits
a faithful action by $\big(\IC^*\big)^3$, but it actually has a much larger
symmetry group.  It is perhaps easiest to see this from the homogeneous
coordinates.  From the equivalence relation \eqref{eq:rescaling}, we see
that we are free to perform an arbitrary $GL(2,\IC)$ transformation on
$(z_1, z_2)$, and an independent one on $(z_3,z_4)$, noting that the
`anti-diagonal' $\IC^*$ subgroup acts trivially.  Finally, there is a $\IZ_2$
action which exchanges $(z_1, z_2)$ with $(z_3, z_4)$.  Altogether,
then, the symmetry group is
\begin{equation} \label{eq:sym_group}
    \frac{\big(GL(2,\IC)\times GL(2,\IC)\big)}{\IC^*}\rtimes \IZ_2~,
\end{equation}
where conjugation by the generator of $\IZ_2$ exchanges the two $GL(2,\IC)$
factors.

\newpage
\section{The \texorpdfstring{$\IZ_n$}{Zn}-hyperconifold singularities}\label{sec:classification}

\subsection{Classification of \texorpdfstring{$\IZ_n$}{Zn}-hyperconifolds}

To determine the class of cyclic hyperconifold singularities, we must find the
finite cyclic subgroups of the conifold's symmetry group that act consistently
with the conditions set out in Section~\ref{sec:intro}.  By the discussion in
Appendix \ref{app:exception}, there is precisely one hyperconifold action that
has a non-trivial projection onto the $\IZ_2$ factor of \eqref{eq:sym_group};
the following theorem describes all the other cases.

\begin{theorem} \label{th:classification}
    With the exception of the single case discussed in Appendix
    \ref{app:exception}, the\linebreak$\IZ_n$-hyperconifold singularities are obtained by
    taking quotients of the conifold by the following actions:
    \begin{equation*}
        (y_1, y_2, y_3, y_4) \mapsto (\z y_1, \z^k y_2, \z^{-k} y_3, \z^{-1}y_4)~,
    \end{equation*}
    where $\z = e^{2\pi \ii/n}$, and $k$ is relatively prime to $n$.  We may denote
    the corresponding\linebreak hyperconifold geometry by $\cC_{n,k}$.
    
    Furthermore, there are isomorphisms
    $\cC_{n,k} \cong \cC_{n,k'}$ for $k' = \pm k^{\pm1} \mod n$.
\end{theorem}
\begin{proof}
The conifold's symmetry group is given by \eqref{eq:sym_group}, and the
hypersurface coordinates parametrise the $(\bf{2},\bf{2})$ representation of
this group.

The content of Lemma \ref{th:exception} is that the exceptional hyperconifold
is the only one for which the quotient group projects non-trivially to $\IZ_2$,
so we can ignore the $\IZ_2$ factor here.  Note that any finite-order element of
$GL(2,\IC)\times GL(2,\IC)$ can be conjugated into the maximal torus (its
Jordan normal form must be diagonal), and therefore can be assumed to
act on the hypersurface coordinates by phases.  Note that the action is such
that the polynomial $p$ is always mapped to a multiple of itself.

The Gorenstein condition is that the canonical divisor be Cartier, i.e. that it
correspond to a line bundle.  As the group action has a fixed point, the
canonical line bundle on $\cC$ will descend to a line bundle on $\cC/\IZ_n$
if and only if the action is trivial on the fibre over the fixed point.  This is
equivalent to the invariance of the trivialising section $\O$, as given in
\eqref{eq:Omega}:
\begin{equation*}
    \O = \oint_{p=0} \frac{dy_1\wedge dy_2 \wedge dy_3 \wedge dy_4}{p} ~.
\end{equation*}
The polynomial $p$ transforms like its constituent monomials $y_1 y_4$ and
$y_2 y_3$, whereas the numerator of this expression transforms like their
product, $y_1 y_2 y_3 y_4$.  Therefore $\O$ transforms like the polynomial
$p$ itself, so we must demand that $p$ be invariant.

Since a hyperconifold is by definition an isolated singularity, the group
must act in such a way that the origin is the only fixed point of any group
element.  Since $\cC$ contains each of the coordinate axes,
this means that each $y_i$ must transform with a \emph{primitive}
$n^{\mathrm{th}}$ root of unity.  Therefore, given any generator of $\IZ_n$,
an appropriate power of it (that also generates the group) must act as stated
in the theorem.

Finally, given $n$, not all relatively prime $k$ give distinct hyperconifolds.  The
coordinate changes $y_2 \leftrightarrow y_3$ and $y_1 \leftrightarrow -y_2 ,
y_3 \leftrightarrow y_4$ each leave $p$ invariant, and implement $k \to -k$
and $k \to k^{-1} \mod n$ respectively, so spaces related by these
transformations are isomorphic.
\end{proof}

The fact that the polynomial $p$ is
invariant under the quotient group in all cases is significant; it means that all
these group actions extend to free actions on the smoothing of the conifold,
given by $p + t = 0$ for $t \neq 0$.  They can therefore all be obtained as limits
of multiply-connected Calabi--Yau manifolds in the straightforward way which
motivated this entire investigation.

The classification of hyperconifolds by two relatively-prime integers, up to the
equivalence given, is reminiscent of the classification of three-dimensional
lens spaces, leading to a simple corollary:
\begin{corollary} \label{th:correspondence}
    There is a one-to-one correspondence between hyperconifolds $\cC_{n,k}$,
    up to isomorphism, and lens spaces $L(n,k)$, up to homeomorphism.
\end{corollary}
We will see in Section~\ref{sec:topology} that this is not a coincidence.

The spaces $\cC_{n,k}$, being quotients of $\cC$ by subgroups of the torus,
are all toric; we give the description of their toric diagrams in the form of a
lemma:
\begin{lemma} \label{lem:diagrams}
    The hyperconifold $\cC_{n,k}$ corresponds to the toric diagram with vertices
    at the points
    $(1,0,0), (1,1,0), (1,k,n), (1,k+1,n)$.
\end{lemma}
\begin{proof}
    The relation between the $y_i$ and the toric coordinates $t_1, t_2, t_3$ is
    given by \eqref{eq:toric_coords}; we can use this to write down the $\IZ_n$
    action, as given by the theorem, in toric terms:
    \begin{equation*}
        t_1 \to t_1~,~ t_2 \to e^{2\p\ii k/n}t_2~,~ t_3 \to e^{-2\p\ii/n}t_3 ~.
    \end{equation*}
    To take the quotient by such a discrete subgroup of the torus, we retain
    the same fan, but subdivide the lattice; for this particular action, the original
    lattice $N$ is replaced by the minimal lattice $N' \supset N$ such that
    $(0,-k/n, 1/n) \in N'$ (see \cite{CLS} for the general result).  An obvious choice of
    basis for the new lattice is $\{(1,0,0), (0,1,0), (0,-k/n, 1/n)\}$.  The fan for the
    conifold consists of a single cone with vertices at $(1,0,0),(1,1,0),(1,0,1),(1,1,1)$,
    and its faces.  Expressed in the new basis, these become
    $(1,0,0), (1,1,0), (1,k,n), (1,k+1,n)$.
\end{proof}
Note that the toric diagram for $\cC_{n,k}$ is not unique, as we can apply any
$SL(3,\IZ)$\linebreak transformation which preserves the hyperplane
in which the vertices lie.  For the purpose of plotting aesthetic diagrams, it helps
to apply a skewing transformation in order to make the diagram `as square as
possible'; two examples are shown in Figure \ref{fig:example_diagrams}.
\begin{figure}
\begin{center}
    \raisebox{3ex}{\includegraphics[width=.2\textwidth]{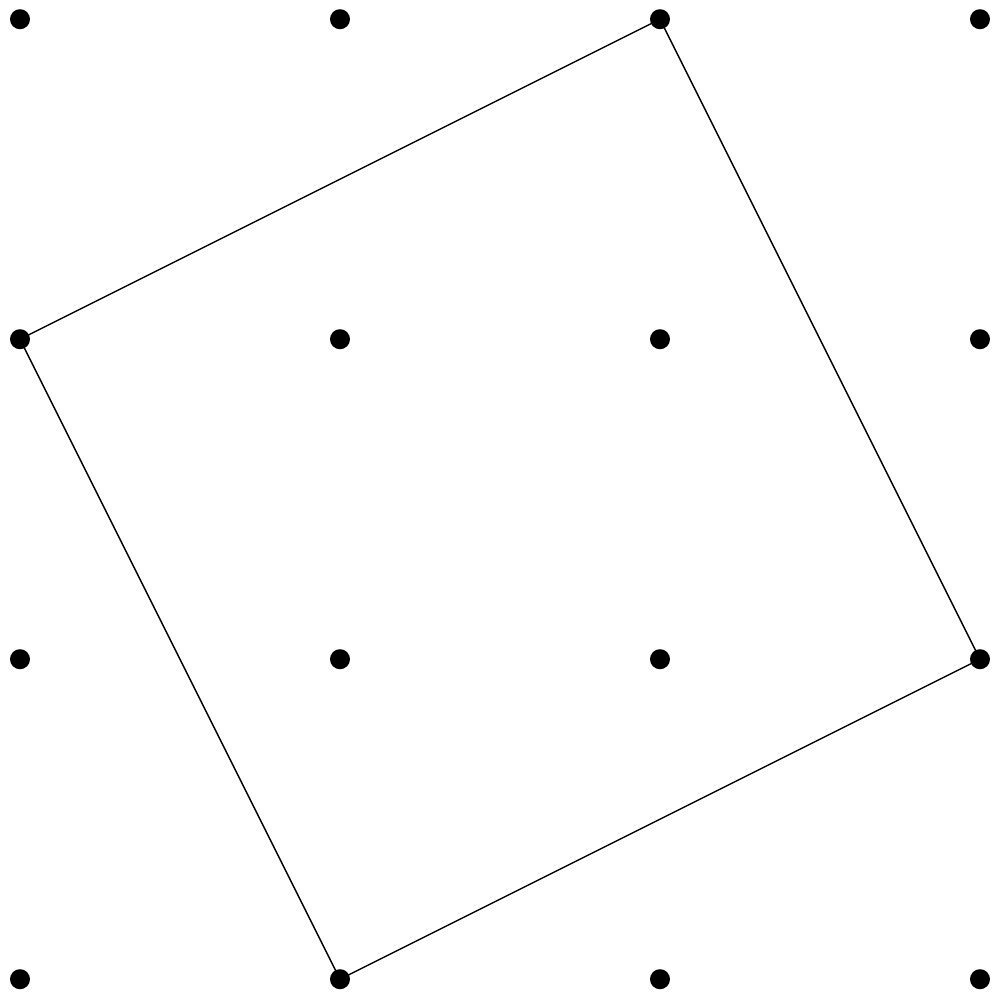}} \hspace{4em}
    \includegraphics[width=.2\textwidth]{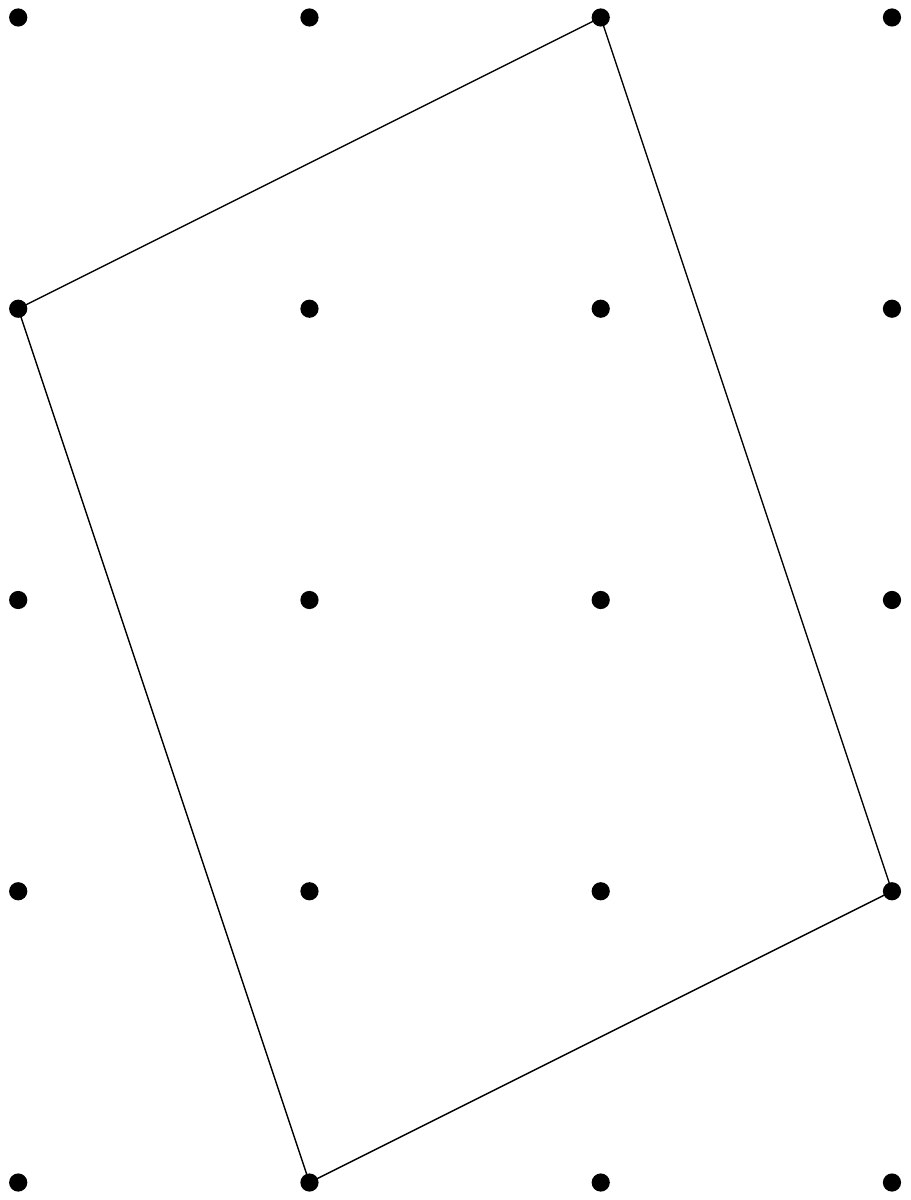}
    \caption{\small \label{fig:example_diagrams}
    Toric diagrams for the hyperconifolds $\cC_{5,2}$ and $\cC_{7,2}$;
    coordinates have been chosen to make the diagrams as clear as
    possible.}
\end{center}
\end{figure}
\begin{remark}
    Interestingly, although many values of $n$ allow inequivalent choices
    of $k$, it seems that in each case, only one value arises in compact examples.
    For instance, I know of a number of compact Calabi--Yau threefolds which
    contain $\cC_{5,2}$ singularities, but none which contain $\cC_{5,1}$.
    Similarly, $\cC_{8,3}$ arises in multiple examples, but to the best of my
    knowledge, $\cC_{8,1}$ does not occur.
\end{remark}
\subsection{Topology} \label{sec:topology}

The conifold $\cC$ is homeomorphic to the cone over $S^3\times S^2$
\cite{Candelas:1989js}.  The vanishing cycle is\linebreak homeomorphic
to $S^3$, i.e., in the smoothed geometry, described by $p + t = 0$ for
$t~\neq~0$, the\linebreak singular point is replaced by a three-sphere $S^3$.  We will
see below that for the\linebreak hyperconifold $\cC_{n,k}$, the vanishing cycle is homeomorphic to
the lens space $L(n,k)$; this is a pleasing geometric origin for the
correspondence described in Corollary \ref{th:correspondence}.

First, however, we require coordinates for the conifold that give us direct
access to its topology.  For this purpose, we will use the nice parametrisation
from \cite{Evslin:2007ux}.  The cone over $S^3\times S^2$ can be
parametrised by triples $(r, X, v)$, where $0 \leq r < \infty$ is a radial
coordinate, $X \in SU(2) \cong S^3$ is an $SU(2)$ matrix representing a
point of $S^3$, and $v \in \IC^2$, $|v| = 1$, represents a point of $\IP^1$
once we impose the equivalence $v \sim e^{i\th}v$ for all $\th \in \IR$.  To
map this space to $\cC$, we first rewrite the defining polynomial as the
determinant of a matrix:
\begin{equation*}
    p = \det W ~,~ W = \left(\begin{array}{l l} y_1 & y_2 \\ y_3 & y_4 \end{array} \right) ~.
\end{equation*}
Now we can write down the following simple map from the cone over
$S^3\times S^2$ to $\cC$:
\begin{equation*}
    W = r X v v^\dagger ~.
\end{equation*}
Note that $r^2 = \Tr(W^\dagger W) = |y_1|^2 + \ldots |y_4|^2$.  This map is shown
to be a homeomorphism in \cite{Evslin:2007ux}.

\begin{theorem}
    The vanishing cycle of $\cC_{n,k}$ is homeomorphic to the lens
    space $L(n,k)$.
\end{theorem}
\begin{proof}
    Recall that the group action giving $\cC_{n,k}$ is
    \begin{equation*}
        (y_1, y_2, y_3, y_4) \mapsto (\z y_1, \z^k y_2, \z^{-k} y_3, \z^{-1}y_4)
            ~,~~\mathrm{where}~ \z = e^{\frac{2\pi\ii}{n}}~.
    \end{equation*}
    This can be realised on the matrix $W$ by
    \begin{equation*}
        \left(\begin{array}{l l} y_1 & y_2 \\ y_3 & y_4 \end{array}\right) \to
            \left(\begin{array}{l l} \z & 0 \\ 0 & \z^{-k} \end{array}\right)
            \left(\begin{array}{l l} y_1 & y_2 \\ y_3 & y_4 \end{array}\right)
            \left(\begin{array}{l l} 1 & 0 \\ 0 & \z^{k-1} \end{array}\right) ~,
    \end{equation*}
    which corresponds to the following action on the $S^3$ and $S^2$ factors
    of the base:
    \begin{equation*}
        X \to \left(\begin{array}{l l} \z & 0 \\ 0 & \z^{-k} \end{array}\right)
        X \left(\begin{array}{l l} 1 & 0 \\ 0 & \z^{k-1} \end{array}\right) ~,~~
        v \to \left(\begin{array}{l l} 1 & 0 \\ 0 & \z^{1-k} \end{array}\right) v ~.
    \end{equation*}
    The action on $S^2$ is simply a rotation (trivial for $k=1$), whereas on
    the $S^3$ factor, it is precisely that action, the quotient by which
    is the lens space $L(n,k)$.
\end{proof}
\subsection{Mirror symmetry}

Mirror symmetry is a deep duality between pairs of Calabi--Yau
threefolds.\footnote{The term `mirror symmetry' is also used more widely,
for similar dualities in different dimensions and involving non-Calabi--Yau
spaces.}  Although\linebreak originally formulated for families of compact
manifolds, it can be generalised to non-compact and singular spaces.  In
\cite{Morrison:1995km}, it was conjectured that (at generic moduli) the mirror
to a\linebreak conifold transition is a conifold transition, and this seems to have
been broadly accepted since.  Although often true, we show here that the
conjecture fails in general.  $\IZ_n$-hyperconifolds are mirror to conifolds with
$n$ nodes, and the related transitions are therefore mirror to each other.
\begin{theorem} \label{th:mirror_sym}
    The mirror of a $\IZ_n$-hyperconifold is nodal, having $n$ ordinary double points.
\end{theorem}
\begin{proof}
    According to local mirror symmetry (following \cite{Gross:2000}), the mirror of a
    toric Calabi--Yau threefold $X$ is given by
    \begin{equation*}
        \{F(u,v,x,y) := uv - f(x,y) = 0 ~|~ (u,v)\in \IC^2 ~,~ (x,y) \in \big(\IC^*\big)^2 \} ~,
    \end{equation*}
where $f$ is a Laurent polynomial, the Newton polygon of which is the toric
diagram for $X$.

The toric diagram for $\cC_{n,k}$ was described in Lemma \ref{lem:diagrams}.
The mirror to $\cC_{n,k}$ itself\linebreak corresponds to the $f$ which is simply the sum
of the four monomials coming from the\linebreak vertices of the parallelogram; any other
choice is mirror to a (possibly partial) resolution.  So we consider the polynomial
\begin{equation*}
    F = uv - 1 - x - x^k y^n - x^{k+1}y^n = uv - (1+x)(1 + x^k y^n) =: uv - f_1 f_2 ~,
\end{equation*}
where we have defined $f_1 = 1+x$ and $f_2 = 1 + x^k y^n$.
This factorisation of $f(x,y)$ follows from the fact that the Newton polygon of
$f$ is a parallelogram.

To find the singularities, we consider $F = dF = 0$.  First, note that $dF = 0$ if
and only if $u = v = 0$, and the remaining equations become
$f_1 f_2 = d(f_1 f_2) = 0$.  There are obviously no solutions to $f_1 = df_1 = 0$,
or to $f_2 = df_2 = 0$, so the only solutions are $f_1 = f_2 = 0$, which occurs
when $x = -1 ~,~ y^n = (-1)^{k+1}$.  This gives $n$ distinct points.

It is simple to check that the Hessian of $F$ is non-degenerate at the singular
points, which are therefore ordinary double points.
\end{proof}

This simple result explains and generalises the examples of \cite{Davies:2011is},
where Batyrev's toric picture of mirror symmetry \cite{Batyrev:1994hm} was used
to show explicitly that hyperconifold transitions between certain families of compact
toric hypersurfaces correspond to conifold transitions between the mirror families.

\section{Singularity resolution and hyperconifold transitions} \label{sec:transitions}

We have seen that hyperconifold singularities occur in compact Calabi--Yau
threefolds; a natural question to ask is whether these admit resolutions which
are smooth Calabi--Yau manifolds.  In \cite{Davies:2009ub}, it was argued that
this is always the case for $\IZ_{2m}$-hyperconifolds, and it was shown to be
true for certain $\IZ_3$ and $\IZ_5$ examples in \cite{Davies:2011is}.  In this
section, I give a general argument that a factorial Calabi--Yau variety
with hyperconifold singularities always admits a Calabi--Yau resolution.

Given a singular Calabi--Yau $X_0$ and a resolution $\pi : \Xh \to X_0$, there are
two conditions to check to determine whether $\Xh$ is also Calabi--Yau: the
resolution must be \emph{crepant} (the canonical class must still vanish), and
\emph{projective} (so that $\Xh$ admits a K\"ahler metric).  The first condition is
fairly trivial, given our toric description of the hyperconifold singularities
in Section~\ref{sec:classification}: any maximal triangulation of the toric diagram
corresponds to a crepant resolution of the singularity.  Checking projectivity is
much harder.

To understand the issue, it is useful to first consider the case of the conifold,
and conifold transitions.  The non-compact space $\cC$ itself is given by
equation \eqref{eq:conifold}:
\begin{equation*}
    y_1 y_4 - y_2 y_3 = 0 ~.
\end{equation*}
This is an example of a non-factorial singularity---the local ring at the origin is
not a unique factorisation domain (UFD), because the reducible element
$y_1 y_4$ has the alternative\linebreak factorisation $y_2 y_3$.  As a result, there are
non-Cartier divisors passing through the singularity; its local divisor class group
is generated by the divisor given by $y_1 = y_2 = 0$.
Blowing up along this divisor gives the well-known small resolution of $\cC$, the
toric diagram for which is shown in Figure \ref{fig:resolved_conifold}.
\begin{figure}[hbt]
\begin{center}
    \includegraphics[width=.2\textwidth]{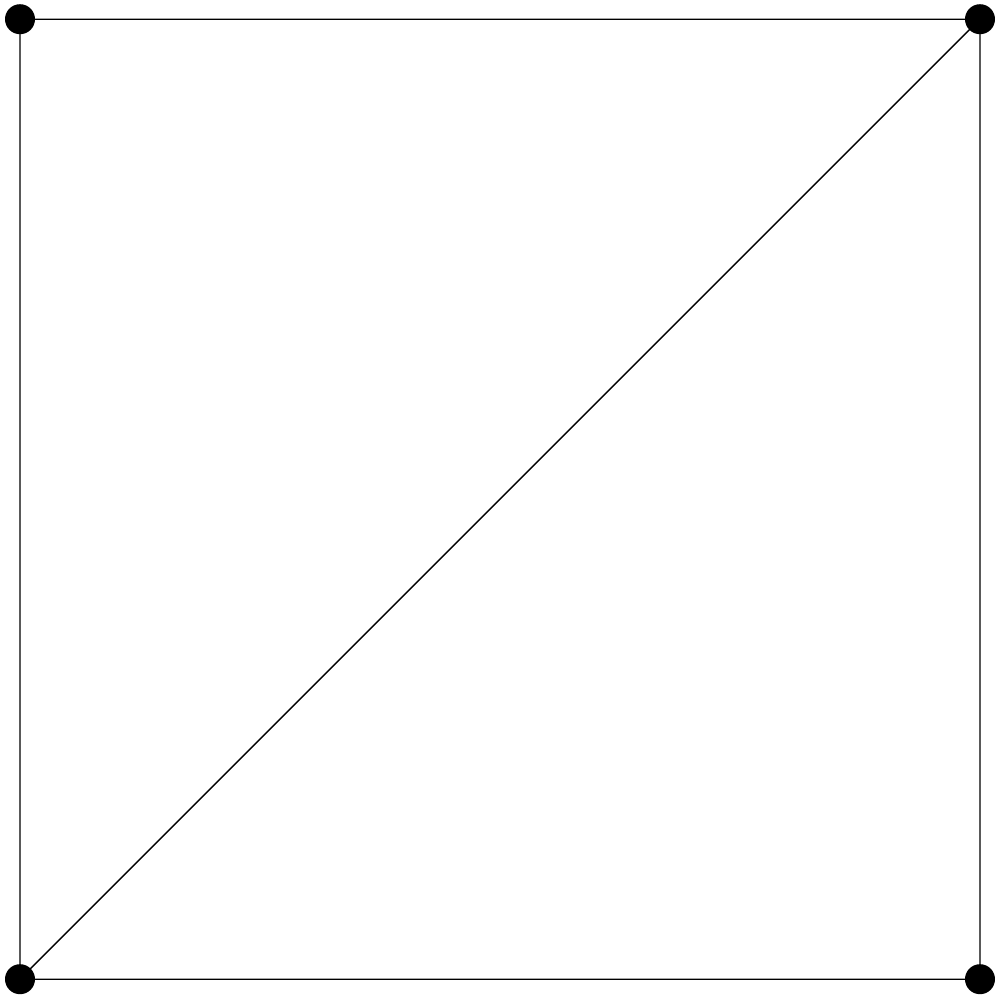}
    \caption{\small \label{fig:resolved_conifold}
    The toric diagram for the small resolution of the conifold.
    }
\end{center}
\end{figure}

Now suppose that some compact Calabi--Yau threefold $X$ contains an
ordinary double\linebreak point $x$.  Then there is some analytic neighbourhood of
$x$ which is isomorphic to a\linebreak neighbourhood of the singular point
of $\cC$.  However, because the local
divisor class group is not an analytic invariant, $X$ does \emph{not}
necessarily have any non-Cartier divisors; $\cO_{X,x}$ may be a UFD.  In
down-to-Earth terms, adding higher-order terms to \eref{eq:conifold} can restore\linebreak
factoriality of the singularity, thereby destroying the non-Cartier divisors.  We can
still perform the small resolution in the analytic category, but the resolved
manifold will not necessarily be an algebraic variety.

The above discussion shows that the existence of a projective crepant
resolution for a nodal threefold is a global issue.  Another way to see this is
to note that all divisors in the resolved conifold are non-compact, and if we are to
try to construct a K\"ahler form (or ample divisor class) on the small resolution of
$X$, we need to know the behaviour of the completions of these divisors in the
compact geometry, in order to calculate intersection numbers.  The existence of
an appropriate resolution must therefore be checked on a case-by-case basis.

The situation for hyperconifolds is very different.  We can show that under
fairly mild conditions, a smoothable Calabi--Yau threefold $X_0$ that contains
just a single hyperconifold singularity is factorial.  We will first need the
following result:
\begin{lemma} \label{lem:node_factoriality}
    Let $Y_t$ be a family of Calabi--Yau threefolds, smooth for $t\neq 0$,
    and let $Y_0$ have a single node.  Then $Y_0$ is factorial.
\end{lemma}
\begin{proof}
    If $Y_0$ were \emph{not} factorial, then there would be a non-Cartier divisor
    passing through the node, and blowing up along this divisor would give a
    projective small resolution.  But $Y_0$ is smoothable by assumption, and
    results of Friedman \cite{Friedman:1986sr,Friedman:1991tt} imply that a
    smoothable nodal Calabi--Yau threefold with a single node does not admit a
    projective small resolution.
    
%
%
\end{proof}

Our motivation for defining and studying hyperconifold singularities was
the possibility of a free group action developing a fixed point.  Adding a
technical assumption on the group structure, we can use the above lemma
to show that the resulting hyperconifold is factorial:
\begin{theorem} \label{th:factoriality}
    Let $X_t = \Xt_t/G$ be a family of Calabi--Yau threefolds, smooth for $t\neq 0$.
    Suppose that $\Xt_0$ is smooth except for a single $G$-orbit of nodes, the image
    of which under the quotient map is an $H$-hyperconifold singularity for some
    subgroup $H \leq G$.  Then if $H$ has a complement in $G$, $X_0$ is factorial.
\end{theorem}
\begin{proof}
    Let $K\leq G$ be a complement of $H$ (this means that every element of
    $G$ can be uniquely written as $hk$, where $h\in H, k\in K$), and let $\xt_0$
    be one of the nodes on $\Xt_0$.  I claim that the nodes of $\Xt_0$ are just
    the $K$-orbit of $\xt_0$.  Indeed, since $X_0$ is smooth except for an
    $H$-hyperconifold, the stabiliser of $\xt_0$ is isomorphic to $H$, and therefore
    by the orbit-stabiliser theorem, the cardinality of $G\xt_0$ is $|G|/|H| = |K|$.
    Now, if we let $k_1 \xt_0 = k_2 \xt_0$, then
    $k_2^{-1}k_1 \in \mathrm{Stab}(\xt_0) = H$.  $K$ is a complement of $H$,
    so $K\cap H$ is just the identity, and therefore $k_1 = k_2$.  We conclude
    that $K\xt_0 = G\xt_0$, and this is exactly the nodes of $\Xt_0$ by assumption.
    
    We deduce from the above that the partial quotient $X'_0 := \Xt_0/K$ is
    smooth except for a single node $x'_0$, which is the image of the nodes on
    $\Xt_0$ under the quotient map.  It is also smoothable---its smoothing is
    given by the $K$-quotient of $\Xt_t$---so by Lemma \ref{lem:node_factoriality},
    $X'_0$ is factorial.
    
    Now note that $\cO_{X'_0,x'_0}$ is a UFD, since $X'_0$ is factorial, and is
    isomorphic to $\cO_{\Xt_0,\xt_0}$ since $K$ acts freely on $\Xt_0$.
    Finally, the local ring at the singular point of $X_0$
    is isomorphic to the $H$-invariant subring of $\cO_{\Xt_0, \xt_0}$, and is
    therefore also a UFD (any subring of a UFD retains this property).  $X_0$
    is smooth elsewhere, so it is factorial. 
\end{proof}

It is therefore natural to restrict ourselves to \emph{factorial} varieties with
hyperconifold\linebreak singularities, and we will see that in this case, the existence
of a projective crepant resolution\linebreak is automatic.  We therefore take the definition
of a hyperconifold \emph{transition} to include\linebreak factoriality of the intermediate
singular variety:

\begin{definition}
    Let $X$ be a smooth Calabi--Yau threefold, which can be deformed to a
    factorial variety $X_0$, which is smooth except for a single
    $\IZ_n$-hyperconifold singularity.  If there is a projective crepant resolution
    $\Xh \to X_0$, then $X$ and $\Xh$ are said to be connected by a
    \emph{hyperconifold transition}, denoted by
    $X \stackrel{\IZ_n}{\rightsquigarrow} \Xh$.
\end{definition}
\begin{remark}
    The factoriality condition will be important for several arguments in the following
    sections.  I suspect that if a smoothable Calabi--Yau threefold is smooth except
    for a single hyperconifold singularity, then it is automatically factorial (as in the
    nodal case), but I do not know how to prove this outside of the restrictive
    assumptions of Theorem \ref{th:factoriality}.  In Section~\ref{sec:examples},
    we will see a smoothable example where factoriality does not hold, but there
    are multiple singularities.  In such cases, the results on resolutions given in
    the following sections do not apply.
\end{remark}
\subsection{Projective crepant resolutions}

Given a smooth family of Calabi--Yau manifolds $X$, degenerating to a factorial
variety $X_0$ with a hyperconifold singularity, we want to show that $X_0$
admits a Calabi--Yau resolution.  The first step is to understand the resolutions
of the non-compact spaces $\cC_{n,k}$ themselves.  Given values of $n$ and
$k$, we seek a resolution which is both \emph{crepant} and \emph{projective}, i.e.,
a smooth space $\widehat{\cC}_{n,k}$ with $K_{\widehat{\cC}_{n,k}} \sim 0$,
and a projective morphism $\pi : \widehat{\cC}_{n,k} \to \cC_{n,k}$.  These
conditions ensure that the resolved space is again Calabi--Yau.
\begin{theorem} \label{th:local_resolution}
    Every hyperconifold $\cC_{n,k}$ has a projective crepant resolution.
\end{theorem}
The spaces $\cC_{n,k}$ are toric, and the proof of this theorem is a
straightforward application of results from \cite{CLS}, which I will repeat here for
convenience.  First, we need the notion of a star subdivision of a fan.
\begin{definition}[\!\!\cite{CLS}, Section~11.1]
    Given a lattice $N$ and a fan $\S$ in the associated real vector space
    $N_{\IR}$, the \emph{star subdivision} of $\S$ with respect to some
    primitive lattice vector $v \in |\S| \cap N$ is denoted by $\S^*(v)$, and is a fan
    consisting of the following cones:
    \begin{enumerate}
        \item
        $\s \in \S$ such that $v \notin \s$.
        \item
        Cones generated by $v$ and $\tau$, where $\tau \in \S$,
        $v \notin \t$, and $\{v\}\cup \tau \subset \s$ for some $\s \in \S$.
    \end{enumerate}
\end{definition}
Roughly speaking, we add the one-dimensional cone generated by $v$, and
subdivide any cones in $\S$ which contain it, in the minimal way which again
yields a fan.  Star subdivisions have the nice property that they give rise to
projective morphisms:
\begin{lemma} \label{lem:star_subdivision}
    Let $\S'$ be the star subdivision of $\S$ with respect to some lattice vector
    $v$.  Then there exists a natural surjective morphism between the corresponding
    toric varieties, $\phi~:~X_{\S'}~\to~X_{\S}$, with the following properties:
    \begin{enumerate}
        \item
        $\phi$ is a projective morphism.
        \item
        If $D_v$ is the toric divisor on $X_{\S'}$ that corresponds to the one-dimensional
        cone generated by $v$, then $-kD_v$ is ample relative to $\phi$ for
        some $k>0$.
    \end{enumerate}
\end{lemma}
\begin{proof}
   See the proof of Proposition 11.1.6 in \cite{CLS}.
\end{proof}
The final concept we need is the multiplicity of a simplicial cone:
\begin{definition}
    Let $\s$ be a $d$-dimensional convex simplicial cone in $N_{\IR} = N\otimes_\IZ \IR$,
    generated by $u_1,\ldots, u_d \in N$.  If $\{e_1,\ldots,e_d\}$ is some basis for $N$,
    and $u_i = \sum_j a_{ij} e_j$, then the \emph{multiplicity} of $\s$ can be defined as
    $\mathrm{mult}(\s) = |\det(a_{ij})|$.
\end{definition}
Obviously $\mathrm{mult}(\s)$ is a positive integer, and $\s$ corresponds to a smooth
toric variety if and only if $\mathrm{mult}(\s) = 1$.  With these preparations in place,
we can now prove Theorem \ref{th:local_resolution}.
\begin{proof}[Proof of Theorem \ref{th:local_resolution}]
    Start with the fan $\S$ for $\cC_{n,k}$; recall from Lemma \ref{lem:diagrams} that
    it consists of a single three-dimensional (3D) cone $\s$, with vertices at
    $(1,0,0), (1,1,0),(1,k,n), (1,k+1,n)$, and its faces.
    
    One can easily check that for each $m \in \IZ$ satisfying $0 < m < n$, there is
    a unique $l \in \IZ^+$ such that $(1,l,m)$ is in the interior of $\s$.  There are therefore
    $n-1$ primitive lattice vectors of this form; label them $v_1,\ldots,v_{n-1}$.  Let
    $\S_1 = \S^*(v_1)$ be the star subdivision of $\S$ with respect to the first of these.
    Then, since $\s$ is a cone over a parallelogram, and $v_1$ lies in its interior,
    $\S_1$ consists of four simplicial 3D cones and their faces (see Figure
    \ref{fig:C52_resolution} for an example).
    
    Now sequentially perform star subdivisions with respect to the vectors
    $v_2,\ldots,v_{n-1}$,\linebreak defining $\S_i = \S_{i-1}^*(v_i)$.  I claim that $\S_{n-1}$ is
    a smooth fan.  To prove this, consider the vector $v_i$ in relation to $\S_{i-1}$
    for $i\geq 2$.
    There are only two possibilities: $v_i$ lies in the relative interior of either a 3D
    cone or a 2D cone in $\S_{i-1}$.  In the former case, we obtain $\S_{i}$ by
    removing the 3D cone and replacing it with three new simplicial 3D cones.  In
    the latter case, the 2D cone containing $v_i$ must be the intersection of two
    3D cones; we obtain $\S_{i}$ by removing them, and replacing each of them
    with two new simplicial 3D cones.  Either way, $\S_{i}$ is again a simplicial
    fan, with the same support as $\S_{i-1}$, but with a total of two more 3D cones.
    
    So $\S_{n-1}$ consists of $4 + 2\times(n-2) = 2n$ simplicial 3D cones and their
    faces.  Furthermore, each 3D cone is generated by primitive lattice vectors
    of the form $(1,l,m)$.  It is a trivial calculation to confirm that the multiplicity of
    such a cone is equal to twice the area of the triangle given by intersecting it
    with the hyperplane on which the vertices lie.  The original cone $\s$
    intersected this hyperplane in a lattice parallelogram of area $n$, and we have
    subdivided it into $2n$ lattice triangles, with areas which must take integral or
    half-integral values.  The only possibility is for each triangle to have area equal
    to $1/2$, so the corresponding cones all have multiplicity one.  Therefore
    $\S_{n-1}$ is a smooth fan.
    
    The resolution constructed here is manifestly crepant, since each new
    primitive lattice vector lies on the same hyperplane as the original vertices.
    It is also projective, because a star subdivision corresponds to a projective
    morphism (Lemma \ref{lem:star_subdivision}), and a composition of
    projective morphisms is projective.
\end{proof}
\begin{remark}
    It follows from Lemma \ref{lem:star_subdivision} that a relatively ample divisor
    for the resolution is given by some negative linear combination of components
    of the exceptional divisor.  These all project to the singular point of $\cC_{n,k}$
    under the resolution map.  This should be contrasted with the case of the small
    resolution of the conifold, for which any relatively ample divisor projects to a
    divisor.
\end{remark}
\begin{figure}
\begin{center}
    \begin{tabular}{c @{\hskip 5em} c}
        \includegraphics[width=.25\textwidth]{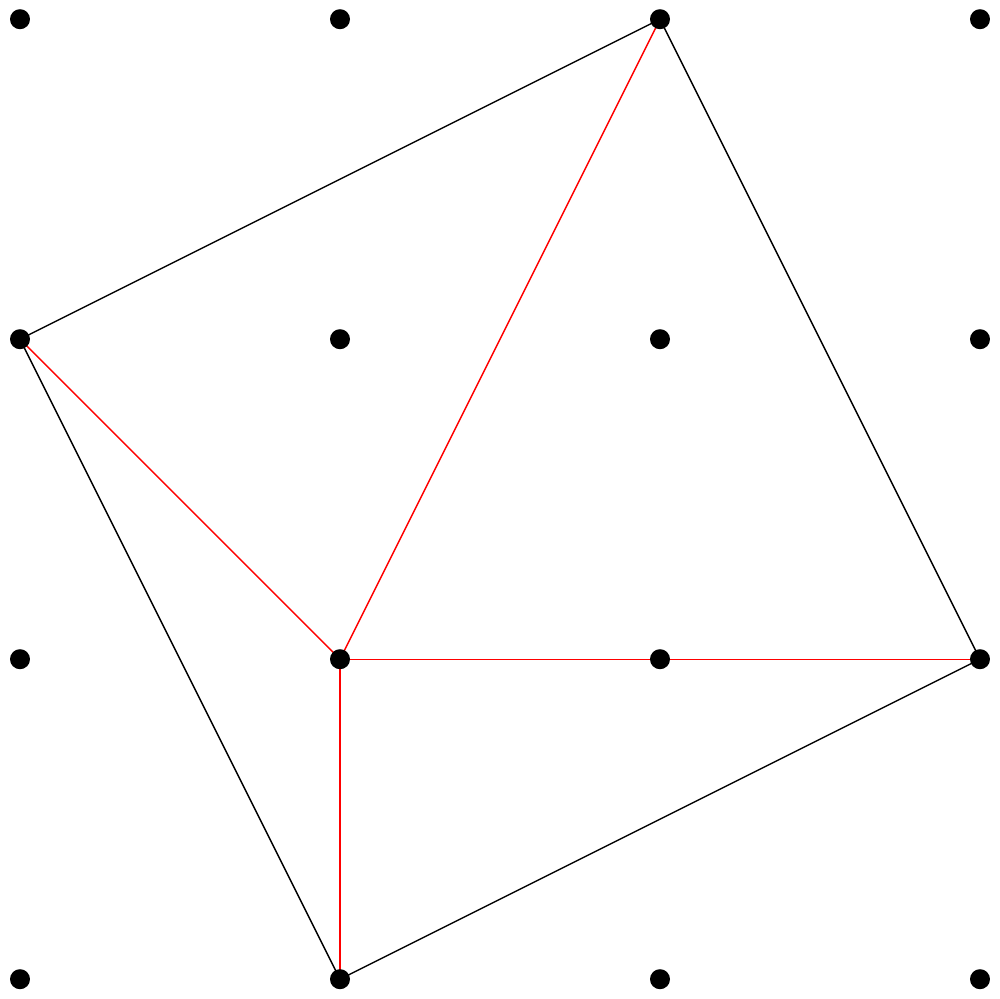} &
            \includegraphics[width=.25\textwidth]{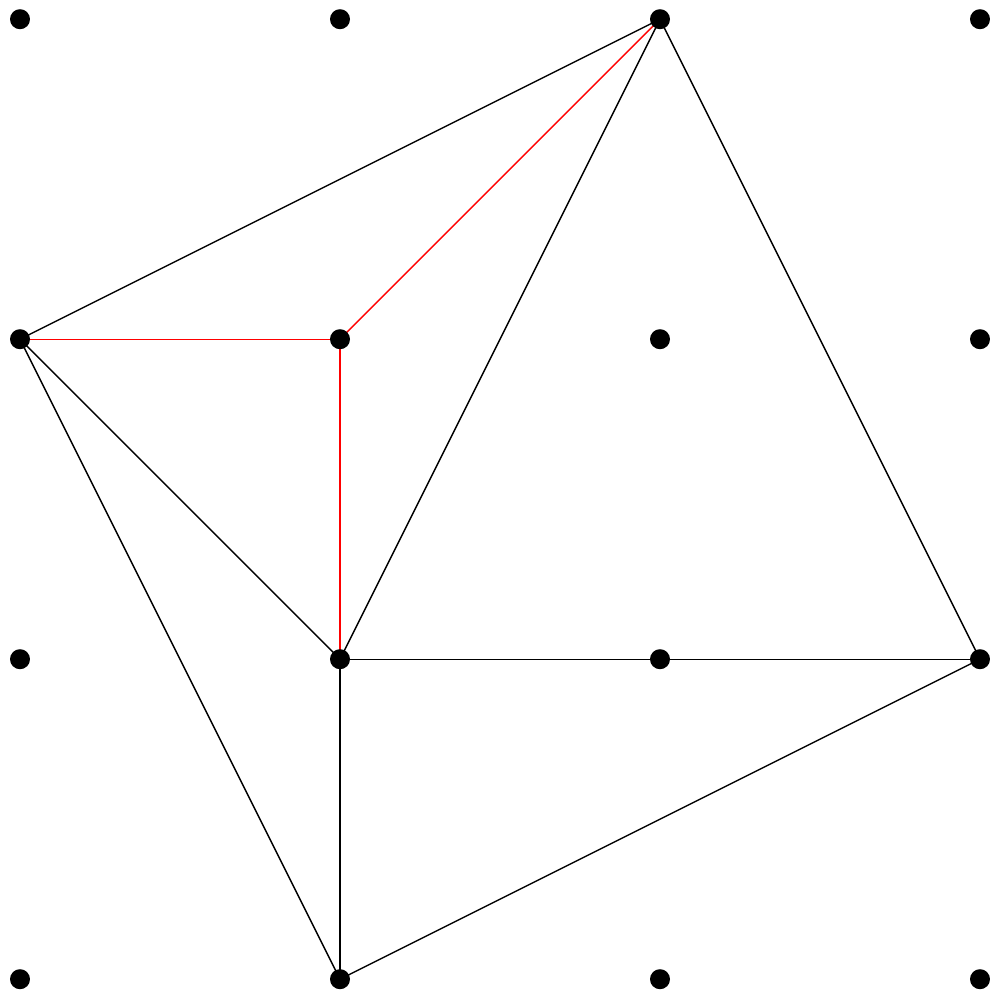}\\[10ex]
        \includegraphics[width=.25\textwidth]{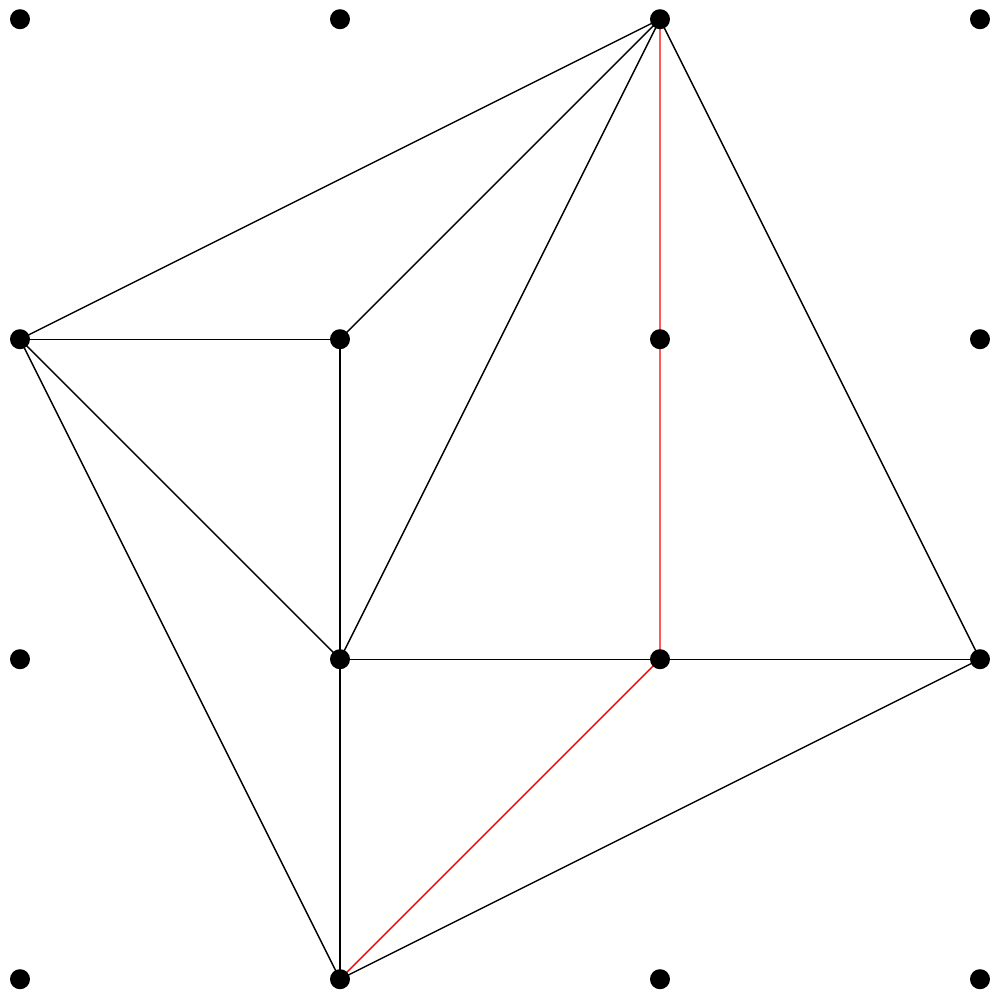} &
            \includegraphics[width=.25\textwidth]{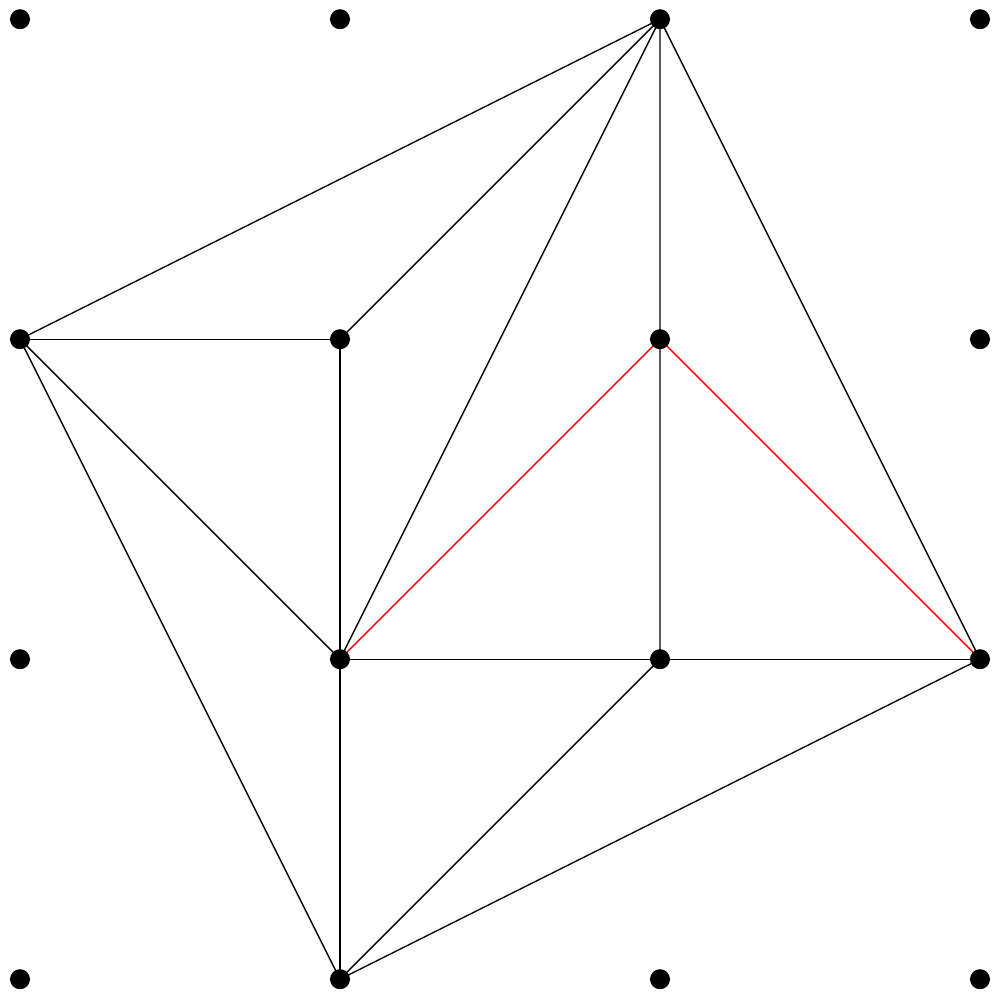}
    \end{tabular}
    \caption{\small \label{fig:C52_resolution}
        An illustration of the resolution of $\cC_{5,2}$, following the process described in
        the proof of Theorem \ref{th:local_resolution}.  In this case it consists of a sequence
        of four star subdivisions.  Note that, once again, the coordinates
        used to make the graphics are different to those used in the text, in order to make
        the diagrams clearer.
    }
\end{center}
\end{figure}

To complete the proof that hyperconifold transitions occur between compact
Calabi--Yau threefolds, we must argue that the local resolutions constructed
above can always be glued into the compact geometries in an appropriate
way (this is the step which can fail for small resolutions of a nodal threefold).

\begin{theorem} \label{th:global_resolution}
    Let $X_0$ be a factorial Calabi--Yau threefold, smooth except for a single\linebreak
    hyperconifold singularity.  Then $X_0$ admits a Calabi--Yau (i.e., projective
    and crepant)\linebreak resolution.
\end{theorem}
\begin{proof}
    The singular point of $X_0$ has an analytic neighbourhood that is isomorphic
    to a neighbourhood of the singular point of $\cC_{n,k}$.  We can therefore
    analytically glue in the exceptional set of some projective resolution
    $\widehat\cC_{n,k}$, and obtain a crepant analytic resolution $\pi : \Xh \to X_0$.
    Being bimeromorphic to a variety, $\Xh$ is an algebraic space.
    
    Recall that a relative ample divisor for the resolution
    $\widehat \cC_{n,k} \to \cC_{n,k}$ can be constructed from components of the
    exceptional divisor $E$; let us choose such a divisor $D$, and abuse
    notation by also denoting by $D$ its image under the embedding
    $E \into \Xh$.  We now use the relative Nakai-Moishezon criterion for
    algebraic spaces \cite{Lazarsfeld,Kollar:1990pr}: $D$ is $\pi$-ample if and
    only if $D^{\dim V}\cdot V > 0$ for all irreducible closed subspaces $V$ for
    which $\pi(V)$ is a point.  The only such $V$ are subspaces of the exceptional
    set $E$, and we know that $D$ is positive on these, since it is a relative ample
    divisor for $\widehat\cC_{n,k} \to \cC_{n,k}$.  Therefore $D$ is also $\pi$-ample,
    so $\Xh$ is projective.
\end{proof}
\subsubsection{The local ample/K\"ahler cone} \label{sec:local_cone}

I will now introduce some terminology which is useful for discussing
resolutions of isolated singularities, including hyperconifolds.
\begin{definition}
    Let $X_0$ be a variety, smooth except at one point, and let $\pi : \Xh \to X_0$
    be a (not necessarily algebraic) resolution map.  Call the exceptional set $E$, and
    denote by $E_1,\ldots,E_N$ its irreducible divisorial components.  Then a
    \emph{local ample divisor} for the resolution is a divisor $H_L = \sum_i t_i E_i$
    such that for any subvariety $Y$ of $E$, $(H_L)^{\dim Y} \cdot Y > 0$.
    The \emph{local ample cone} is the cone generated by numerical equivalence
    classes of local ample divisors; this may be empty.
\end{definition}
A local ample divisor for a resolution $\pi$ is therefore simply a $\pi$-ample
divisor that is supported only on the exceptional set; indeed, the definition comes
from the relative Nakai-Moishezon criterion.
\begin{remark}
    In the case where $\Xh$ is Calabi--Yau, ample ($\IR$)-divisor classes are
    equivalent to K\"ahler classes, and we can equivalently talk about the
    \emph{local K\"ahler cone}; this language is more natural for applications to
    string theory.  The word `local' has a double meaning: the local K\"ahler
    cone depends only on the fibre $E$ over the singular point of $X_0$, and it also
    describes the geometry of the K\"ahler cone in a neighbourhood of the face
    that corresponds to contracting $E$ to a point.
\end{remark}
\begin{lemma}
    Let $X_0$ be a factorial projective variety with a unique singular point, and
    $\pi : \Xh \to X_0$ an analytic resolution.  Then $\Xh$ is projective if and only if
    $\pi$ has a non-empty local ample cone.
\end{lemma}
\begin{proof}
     One direction is easy: we have already noted that if $H_L$ is a local ample
     divisor, it is a $\pi$-ample divisor.  Then if we choose some ample divisor
     $H_0$ on $X_0$, $m\,\pi^*H_0 + H_L$ will be ample on $\Xh$ for all $m\gg 0$
     \cite{Lazarsfeld}.
     
     Conversely, suppose $\Xh$ admits an ample divisor $H$.  Then we can write
     $H = \pi^*(H_0) + H_L$, where $H_0$ is some class on $X_0$, and $H_L$ is
     supported only on the exceptional set.  Since $X_0$ is factorial, $H_0$ has a
     representative not supported at the singular point (see Lemma
     \ref{lem:moving_support}); it follows that $\pi^*(H_0)$ is zero on the exceptional
     set, and therefore, since $H$ is ample, $H_L$ must be a local ample divisor by
     the Nakai-Moishezon criterion.
\end{proof}
\begin{remark}
    When the singularity on $X_0$ is toric, and the resolution is also toric, there
    is a simpler condition for a local ample divisor: in this case, $H_L$ is a local
    ample divisor if and only if $H_L \cdot C > 0$ for all toric curves $C$ in the
    exceptional set.  This follows immediately from the relative `toric Kleiman
    condition' \cite{CLS}.
\end{remark}

\newpage
\begin{example}[Resolving $\cC_{3,1}$] \label{ex:C31_resolutions}
    There are two distinct crepant resolutions of $\cC_{3,1}$, which is the
    unique $\IZ_3$-hyperconifold; their toric diagrams are shown in
    Figure~\ref{fig:C31_res}.  We will see here that only the first has a
    non-empty local ample/K\"ahler cone, so any factorial projective variety
    containing this singularity admits exactly one crepant, projective resolution.
    \begin{figure}[ht]
    \begin{center}
        \includegraphics[width=.2\textwidth]{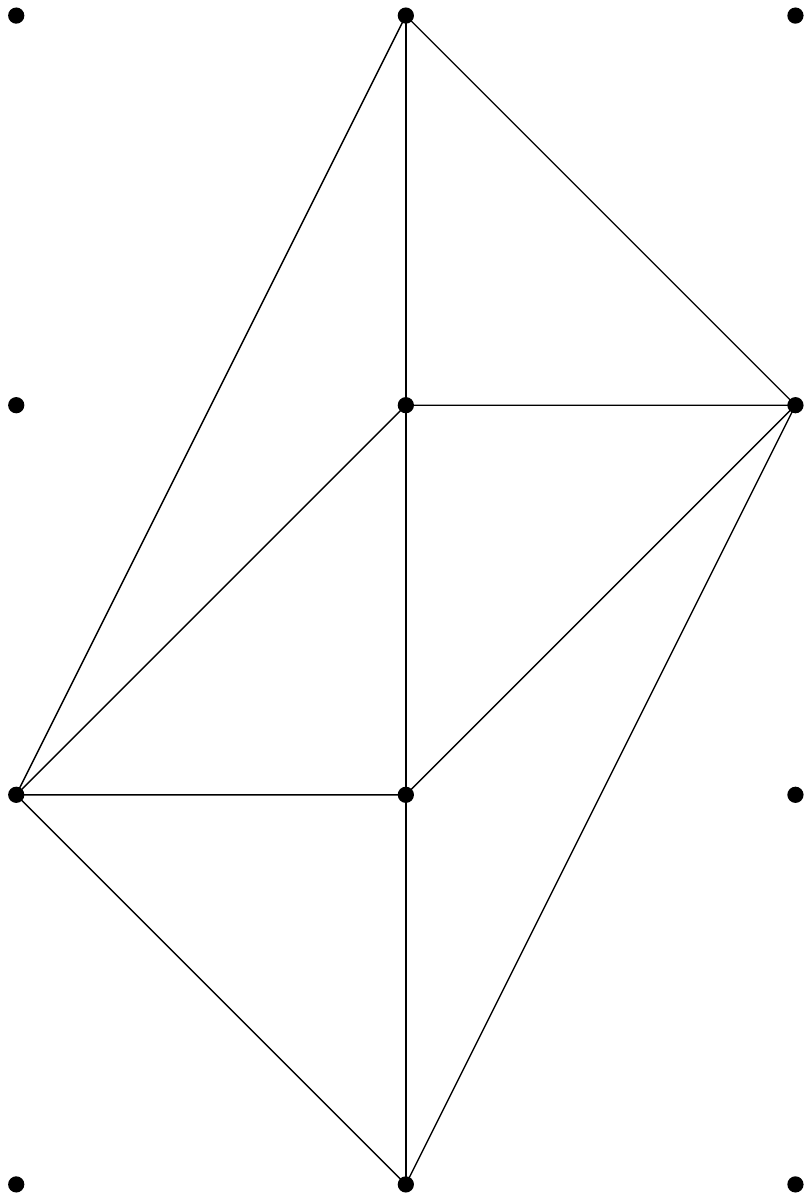}
        \hspace{5em}
        \includegraphics[width=.2\textwidth]{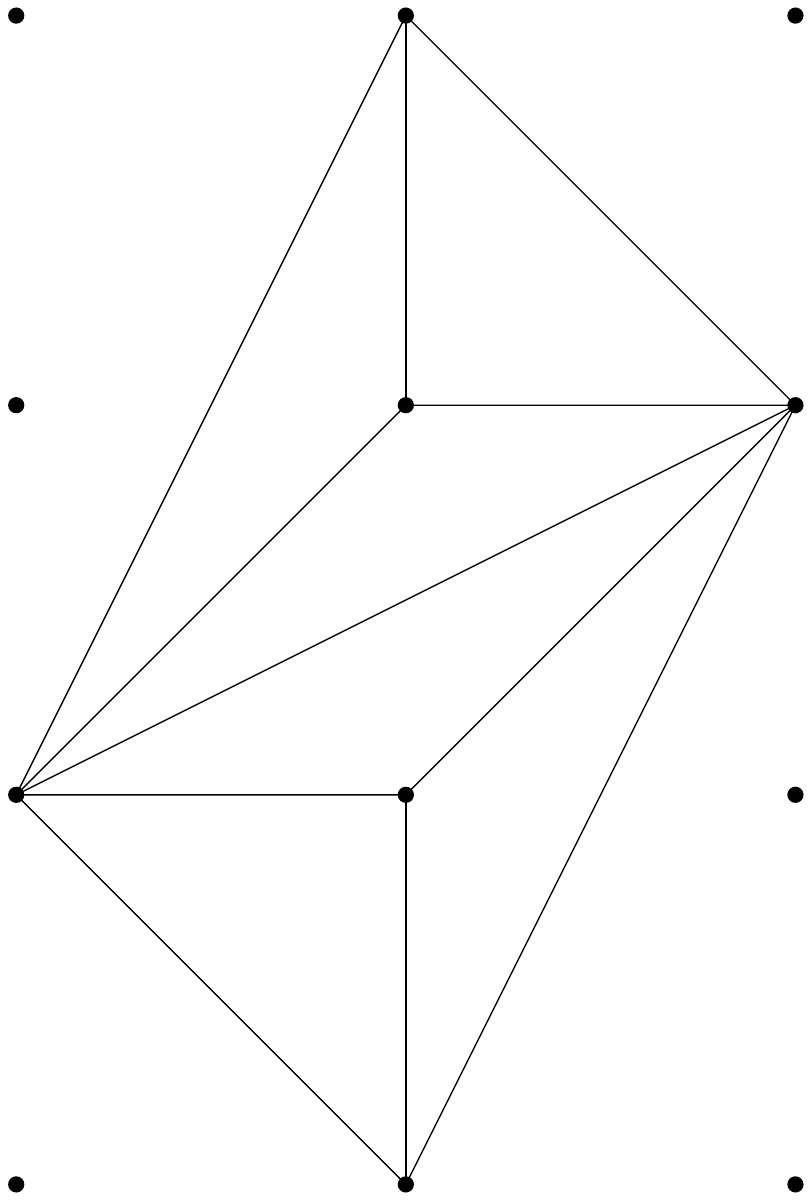}
        \caption{\small \label{fig:C31_res}        
        The two crepant resolutions of $\cC_{3,1}$.  The first has a non-empty local
        K\"ahler cone, while the second does not.
        }
    \end{center}
    \end{figure}

    \noindent\underline{Case 1}\\
    Note that the first resolution corresponds to star subdivisions with respect to first
    one, then the other, of the interior points in the toric diagram for $\cC_{3,1}$; we
    therefore expect a non-empty local ample cone.  The exceptional set consists
    of two copies of the Hirzebruch surface $F_1$, glued along a common copy of
    $\IP^1$ (this can be seen by utilising the `star construction' of toric geometry
    \cite{Fulton}).  Call the first surface $S_1$, and the second $S_2$, and define a
    putative local ample divisor $D_L = t_1 S_1 + t_2 S_2$.
    
    We can now evaluate the divisor $D_L$ against various curves in the resolution.
    To do so, note that by adjunction, and the fact that the ambient space
    $\widehat \cC_{3,1}$ is Calabi--Yau, we have $K_{S_i} \sim S_i\vert_{S_i}$, so
    that if a curve $C$ is contained in $S_i$, we get
    $S_i \cdot C = K_{S_i}\cdot C = -2 - C^2$, where the last equality follows from
    the fact that $S_i$ is a toric surface.
    
    Each surface $S_i$ contains toric curves of self-intersection $0$, which intersect
    the other surface transversely at a single point.  For such a curve $C_1 \subset S_1$,
    we get $D_L\cdot C_1 = -2 t_1 + t_2$, and similarly for $C_2 \subset S_2$, we get
    $D_L \cdot C_2 = t_1 - 2t_2$.  Therefore $-2t_1 + t_2 > 0$ and $t_1 - 2t_2 > 0$ are
    necessary conditions for a local ample divisor.  The remaining toric curves only yield
    weaker inequalities, so the local ample/K\"ahler cone is given by
    \begin{equation*}
        2 t_1 < t_2 < \frac 12 t_1 ~.
    \end{equation*}
    Note that this implies $t_1 < 0, t_2 < 0$.
    
    \newpage
    \noindent\underline{Case 2}\\
    The second triangulation has an exceptional set consisting of two disjoint
    surfaces $S_1, S_2$, each isomorphic to $\IP^2$, and a rational curve $C$
    which intersects each transversely at a single point.  From considering the
    toric curves embedded in either $S_1$ or $S_2$, we get the conditions
    $t_1 < 0, t_2 < 0$.  However, $D_L\cdot C = t_1 + t_2$, so in this case the
    local ample/K\"ahler cone is empty.
\end{example}
\subsection{Topological data}

Given a hyperconifold transition $X \stackrel{\IZ_n}{\rightsquigarrow} \Xh$, the
results of this section will allow us to calculate the topological data of $\Xh$ in
terms of that of $X$.

\subsubsection{Hodge numbers} \label{sec:hodge_nums}

Smooth Calabi--Yau threefolds have only two independent Hodge numbers
undetermined by the Calabi--Yau conditions, which we can take to be
$h^{1,1}$ and $h^{2,1}$.  These quantities behave simply under
hyperconifold transitions.
\begin{theorem} \label{th:hodgenos}
    If $X \stackrel{\IZ_n}{\rightsquigarrow} \Xh$ is a $\IZ_n$-hyperconifold transition, then the Hodge
    numbers of $X$ and $\Xh$ are related by
    \begin{align*}
        h^{1,1}(\Xh) &= h^{1,1}(X) + n-1 ~,\\
        h^{2,1}(\Xh) &= h^{2,1}(X) - 1 ~.
    \end{align*}
\end{theorem}
\begin{proof}
    Due to the restricted Hodge diamond of a Calabi--Yau threefold, the topological
    Euler number is given simply by
    \begin{equation*}
        \ch(X) = 2\big(h^{1,1}(X) - h^{2,1}(X)\big) ~.
    \end{equation*}
    The change in the Euler number is easy to calculate via the surgery picture
    of a transition.  We delete a three-cycle homeomorphic to a lens space, which
    has $\chi = 0$, and replace it with one of the toric spaces described
    in Section~\ref{sec:transitions}.  The Euler number of a toric variety is equal to
    the number of top-dimensional cones in its fan; it follows that for any complete
    crepant resolution of a $\IZ_n$-hyperconifold, the Euler number increases by
    $2n$.  In terms of Hodge numbers,
    \begin{equation} \label{eq:Euler_change}
        \big(h^{1,1}(\Xh) - h^{1,1}(X)\big) - \big(h^{2,1}(\Xh) - h^{2,1}(X)\big)
        = 
        n ~.
    \end{equation}
    Secondly, for a smooth Calabi--Yau threefold, the Hodge number $h^{1,1}$ is
    equal to the rank of the divisor class group, which allows us to calculate the
    first term.  The intermediate singular variety $X_0$ is factorial by assumption,
    so its class group is isomorphic to that of $X$ (no new divisor classes are
    created on the singular variety).  The resolution introduces an exceptional
    divisor with $n-1$ components, all linearly independent.  Therefore
    \begin{equation*}
        h^{1,1}(\Xh) - h^{1,1}(X) = n-1 ~.
    \end{equation*}
    Combining this with \eqref{eq:Euler_change} gives the claimed relations.
\end{proof}
The decrease of $h^{2,1}$ by one can be understood as imposing one condition
on the moduli of $X$ in order to form the hyperconifold singularity.  We will see
this very explicitly in Example~\ref{ex:explicit_transition}.

\subsubsection{The fundamental group}

In the simplest case, a hyperconifold singularity develops when a generically-free
group action develops at least one fixed point, so the corresponding singular variety
will have a smaller fundamental group than the smooth members of the
family.  As we saw above, the resolution process involves gluing in a complete
(reducible) toric surface, and these are simply-connected, so a hyperconifold
transition will typically reduce the size of the fundamental group.

We will consider the situation where some general finite group $G$ acts freely
on a Calabi--Yau threefold $\Xt$, which can be deformed until a cyclic subgroup
develops a fixed point.  Let $g \in G$ be a generator of this subgroup, and suppose
$\xt \in \Xt$ is the fixed point.  Then for any $h \in G$, we have
$\big(hg^k h^{-1}\big)\big(h(\xt)\big) =h(\xt)$, so all elements conjugate to $g$, or any of
its powers, also necessarily develop fixed points, which fill out the $G$-orbit of $\xt$.
We might therefore expect the hyperconifold transition to `destroy' the smallest normal
subgroup containing $\spn{g}$---the normal closure $\spn{g}^G$.  This can be made
precise:
\begin{theorem} \label{th:pi1}
    Let $t$ be a parameter, and $X_t = \Xt_t/G$ be a family of Calabi--Yau threefolds,
    smooth for $t\neq 0$, where each $\Xt_t$ is simply-connected.  Suppose that
    $\Xt_0$ is smooth apart from a single $G$-orbit of nodes, and $X_t$ is smooth
    apart from a $\IZ_n$-hyperconifold singularity, which is the image of the node(s)
    under the quotient map.  If $\Xh$ is a Calabi--Yau resolution of $X_0$, then
    $\pi_1(\Xh) \cong G/N$, where $N$ is the normal closure of the set of elements
    of $G$ which have fixed points on $\Xt_0$.
\end{theorem}
\begin{proof}
    Let $\xt \in \Xt_0$ be one of the nodes.  Then $\Yt = \Xt_0\setminus G\xt$ is
    smooth by assumption, and $G$ acts freely on it, yielding a smooth quotient
    space $Y = \Yt/G$ with $\pi_1(Y) \cong G$.  Note that, since all the nodes on $\Xt$
    are identified by $G$, the space $Y$ is obtained from $X_0$ by
    removing the unique singular point.
    
    Topologically, the resolution corresponds to gluing in a contractible neighbourhood
    $\Th$ of the exceptional set, so that $\Xh = Y \cup \Th$.  The result is a simple
    application of the Seifert-van Kampen theorem \cite{Hatcher} to this decomposition,
    as I will now outline.  Some neighbourhood of a $\IZ_n$-hyperconifold singularity is
    homeomorphic to a cone over $S^3\times S^2/\IZ_n$, which after deleting the tip
    is homotopy equivalent to $S^3\times S^2/\IZ_n$; therefore
    $Y \cap\Th \simeq S^3\times S^2/\IZ_n$.  Finally, note that $\Th$ is
    at least homotopy-equivalent to a toric variety, corresponding to the resolutions
    constructed previously.  Since the fans for these all contain at least one
    full-dimensional cone, they are all simply-connected \cite{Fulton}.  So the data we
    have is
    \begin{equation*}
        \hat X = Y\cup\Th ~,~ \pi_1(Y) \cong G ~,~ \pi_1(\Th) \cong \mb{1} ~,~ \pi_1(Y\cap \Th) \cong \IZ_n~.
    \end{equation*}
    Since $Y$, $\Th$ and $Y\cap \Th$ are all path-connected, the result follows
    immediately from the Seifert-van Kampen theorem.
\end{proof}
There are examples of hyperconifold transitions which do not arise in exactly
the way described in the Lemma.  For instance, $X$ might already be
simply-connected, yet undergo a hyperconifold transition to another
simply-connected space; see \cite{Davies:2011is} for some examples.
The fundamental group can nevertheless always be computed using the simple
technique of the proof above.

\subsubsection{Intersection numbers} \label{sec:intersection_numbers}

Finally, we can calculate the triple-intersection numbers of $\Xh$ in terms of
those of $X$.  To do so, we will choose a basis consisting of divisors pulled
back from $X_0$, and components of the exceptional divisor.  The calculation
of intersection numbers is made simple by the following lemma, which shows
that we can choose a basis for the class group of $X_0$ which consists of
divisors which do not intersect the singularity.
\begin{lemma}[See, e.g., \cite{Shafarevich}] \label{lem:moving_support}
    Let $X_0$ be a factorial quasi-projective variety, and $\{x_1,\ldots,x_m\}$ be
    some finite collection of points in $X_0$.  Then, given any divisor $D$ on
    $X_0$, there exists a divisor $D' \sim D$ which is not supported at any of
    the $x_i$.
\end{lemma}
%
%

We will now define a convenient basis of divisors in terms of which to compute
the intersection numbers.  Let $I,J,\ldots$ be indices running from $1$ to
$h^{1,1}(\Xh)$, which we know from Section~\sref{sec:hodge_nums} is equal to
$h^{1,1}(X)+n-1$.  We will split the basis into two parts; to this end, let the
lower case latin indices $i,j,\ldots$ run from $1$ to $h^{1,1}(X)$,
and greek indices $\a,\b,\ldots$ run from $h^{1,1}(X)+1$ to $h^{1,1}(X)+n-1$.

Take a basis $\{D_i\}$ for the divisor class group of $X_0$, and let
$\{\hat D_i\}$ be their pullbacks to $\Xh$ under the resolution map.  Then
arbitrarily order the components of the exceptional divisor, and denote these
by $\{\hat D_\a\}$.  The total collection $\{\hat D_I\}$ defines a basis on $\Xh$.

It is now straightforward to calculate all the intersection numbers of $\Xh$:
\begin{theorem}
    Let $X \stackrel{\IZ_n}{\rightsquigarrow} \Xh$ be a hyperconifold transition,
    with intermediate singular variety $X_0$.  Then, in the basis for $H^{1,1}(X)$
    defined above, the triple-intersection numbers $\hat d_{IJK}$ are
    \begin{enumerate}
        \item
            $\hat d_{ijk} = d_{ijk}$, where $d_{ijk}$ are the triple-intersection numbers
            on $X$ (and $X_0$).
        \item
            $\hat d_{ij\a} = \hat d_{i\a\b} = 0$
        \item
            $\hat d_{\a\b\g}$ are the intersection numbers of the components of the
            exceptional divisor of $\widehat\cC_{n,k}$, and can be calculated using standard
            toric techniques.
    \end{enumerate}
\end{theorem}
\begin{proof}
    By Lemma \ref{lem:moving_support}, the divisors $D_i$ can be chosen to miss
    the singularity of $X_0$.  Since $\Xh \cong X_0$ away from the exceptional
    divisor/singular point, the intersection numbers of the $\hat D_i$ are obviously
    the same as those of the $D_i$.  To see that these are the same as the
    intersection numbers on $X$ (the smoothing of $X_0$), note that $X_0$ with the
    singularity deleted deforms to $X$ with a lens space deleted, and the divisors
    deform with the whole space.  Intersection numbers are topological on smooth
    spaces, so $\hat d_{ijk} = d_{ijk}$.
    
    If we choose the $D_i$ to miss the singularity as above, then it is clear that
    $\hat D_i \cap \hat D_\a = \emptysetnew ~\forall~ i,\a$, which immediately implies
    $\hat d_{ij\a} = \hat d_{i\a\b} = 0$.
    
    The final statement is obvious, as intersection numbers are topological, and
    the resolved toric space $\widehat\cC_{n,k}$ is homeomorphic to a neighbourhood of the
    exceptional divisor on $\Xh$, of which the $\hat D_\a$ are components.
\end{proof}
%

\section{Examples and applications} \label{sec:examples}

Hyperconifold transitions provide a method to construct new Calabi--Yau
threefolds from existing ones.  In fact, we already have one example for
free: we saw in Section~\ref{sec:intro} that the $\IZ_5$ quotient of the
quintic can develop a $\IZ_5$-hyperconifold singularity, and the results of
Section~\ref{sec:transitions} imply that this has a Calabi--Yau resolution $\Xh$,
and allow us to calculate its topological data.  The $\IZ_5$ quotient of the
quintic has Hodge numbers $\hodgenos = (1,21)$, and fundamental group
$\IZ_5$, which is destroyed by the transition.  Therefore
$\hodgenos(\Xh) = (5,20)$ and $\pi_1(\Xh) \cong \id$.  This example was
already given in \cite{Davies:2011is}, and is implicitly included in earlier
work showing that geometric transitions connect all Calabi--Yau threefolds
which are hypersurfaces in toric fourfolds \cite{Kreuzer:2000xy}.

In this section, I describe some examples of hyperconifold transitions,
chosen to\linebreak complement those previously given in \cite{Davies:2011is}.  In
that paper, examples were given where the ambient space has multiple
fixed points, allowing for `chains' of hyperconifold transitions linking
multiple manifolds, and a hyperconifold transition was also used to
construct the first Calabi--Yau threefold with fundamental group $S_3$
(the permutation group on three\linebreak letters).  I will also present two examples
of hyperconifold singularities embedded in compact Calabi--Yau
threefolds which are not factorial; in these cases, the singular varieties
do not occur as the intermediate point of a hyperconifold \emph{transition},
at least as defined in this paper.

\begin{example}[$X^{1,3} \stackrel{\IZ_{10}}{\rightsquigarrow}X^{10,2}$] \label{ex:explicit_transition}
    Our first example will be very typical, and I will describe each step of the
    analysis in detail; this can then be used as a blueprint for any attempt to
    construct new Calabi--Yau manifolds via hyperconifold transitions.
    
    We begin with a manifold $X$ with fundamental group $\IZ_{10}\times\IZ_2$,
    first constructed in \cite{Candelas:2008wb}.  The covering space, $\Xt$, is
    embeddable in a product of five copies of $\IP^1$ as the complete intersection
    of two `quintilinear hypersurfaces'; its configuration matrix \cite{Candelas:1987kf}
    is
    \begin{equation*}
        \cicy{\IP^1\\ \IP^1\\ \IP^1\\ \IP^1\\ \IP^1}{1 & 1 \\1 & 1 \\1 & 1 \\1 & 1 \\1 & 1} ~.
    \end{equation*}
    Denote the two quintilinear polynomials by $p_1, p_2$.  If we take homogeneous
    coordinates $t_{i,a}$ on the ambient space, where $i \in \IZ_5$ and $a \in \IZ_2$,
    then we can define an action of the group $\IZ_{10}\times\IZ_2$ on the ambient
    space by allowing the generators of the two factor groups to act as\footnote{This
    action does not correspond to a linear representation of the group, but rather to a
    projective representation; the generators commute only up to rescaling of the
    homogeneous coordinates.}
    \begin{align} \label{eq:Z10Z2action}
        g_{10} &: t_{i,a} \mapsto t_{i+1,a+1} ~;\\ \notag
        g_2 &: t_{i,a} \mapsto (-1)^a t_{i,a} ~.
    \end{align}
    If we demand that the polynomials transform equivariantly under this action,
    then the group will act on the corresponding manifolds $\Xt$; an appropriate
    choice of transformation rule is
    \begin{align*}
        g_{10} &: p_1 \mapsto p_2 ~,~ p_2 \mapsto p_1 ~;\\
        g_2 &: p_1 \mapsto p_1 ~,~ p_2 \mapsto -p_2 ~.
    \end{align*}
    To write down the most general polynomials satisfying these properties, it is
    convenient to first define the following quantities, which are invariant under
    the $\IZ_5$ subgroup generated by $g_{10}^2$:
    \begin{equation*}
        m_{abcde} := \sum_i t_{i,a}t_{i+1,b}t_{i+2,c}t_{i+3,d}t_{i+4,e}~.
    \end{equation*}
    We now take the polynomials $p_1,p_2$ to be
    \begin{align*}
        p_1 &= \frac{A_0}{5} m_{00000} + A_1 m_{00011} + A_2 m_{00101} + A_3 m_{01111} ~,\\[2ex]
        p_2 &= \frac{A_0}{5} m_{11111} + A_1 m_{11100} + A_2 m_{11010} + A_3 m_{10000} ~,
    \end{align*}
    where the coefficients are arbitrary complex numbers.  It can be checked that
    even these very restricted polynomials cut out a smooth manifold for generic
    choices of the coefficients; I will defer to \cite{Candelas:2008wb} for the
    details of this step.
    
    To establish the existence of a smooth quotient, we must also check that the
    group action on $\Xt$ is fixed-point-free.\footnote{If $X = \Xt/G$ is smooth, and
    $\Xt$ is a Calabi--Yau threefold, then $X$ is automatically another Calabi--Yau
    manifold.  See \cite{Davies:2011fr} for example.} To do so, it is not necessary
    to check separately that \emph{every} group element acts freely, because if
    $g \in G$ has a fixed point, then so do all its powers.  A general element of
    $\IZ_{10}\times\IZ_2$ can be written as $g_{10}^k g_2^a$, with
    $0\leq k \leq 9$ and $0\leq a \leq 1$.  Depending on the values of $k$ and $a$,
    some power of this will be $g_{10}^2, g_{10}^5, g_{10}^5 g_2$, or $g_2$,
    so we need only check that these elements act freely.  The fixed points of these
    elements are listed in the following table:
    \begin{center}
    \begin{tabular}{|c|c|}
        \hline
        Element & Fixed points \\
        \hline\str
        $g_{10}^2$ & $[t_{i,0} : t_{i,1}] = [t_{0,0} : t_{0,1}] ~\forall~i$ \\
        \hline\str
        $g_{10}^5$ & $t_{i,a} = \pm t_{i,a+1}$ (sign can depend arbitrarily on $i$) \\
        \hline\str
        $g_{10}^5 g_2$ & $t_{i,a} = \pm \ii t_{i,a+1}$ (sign can depend arbitrarily on $i$) \\
        \hline\str
        $g_2$ & For each $i$, either $t_{i,0} = 0$ or $t_{i,1} = 0$ \\
        \hline
    \end{tabular}
    \end{center}
    In the case of $g_{10}^2$, the fixed points form an embedded copy of $\IP^1$,
    and the two polynomials become quintics in this $\IP^1$, which do not have
    simultaneous solutions for arbitrary coefficients.  The fixed points of the other
    elements are all isolated, and at each one, the polynomials $p_1, p_2$
    evaluate to non-trivial linear combinations of the coefficients.  We conclude
    that for generic choices of the coefficients, the group action is free.  We
    therefore obtain a smooth Calabi--Yau quotient, $X^{1,3}$, with fundamental
    group $\IZ_{10}\times\IZ_2$.  The three complex structure moduli correspond
    to the three degrees of freedom in the coefficients of $p_1,p_2$, up to scale.
    
    We seek a $\IZ_{10}$-hyperconifold transition, so wish to specialise the choice
    of polynomials so that $\Xt$ contains a fixed point of the $\IZ_{10}$ subgroup.
    Referring to the table above, and noting that some point $\xt \in \Xt$ is fixed by
    $g_{10}$ if and only if it is fixed by $g_{10}^2$ and $g_{10}^5$, we see that
    there are exactly two such points, given by
    $[t_{i,0} : t_{i,1}] = [1 : \pm 1]$ for all $i$ (with the same sign chosen for each
    value of $i$).  These are exchanged by $g_2$, so the symmetric $\Xt$ will
    contain one if and only if it contains the other.  At these points, we have
    $p_1 = \pm p_2 = A_0 + 5A_1 + 5A_2 + 5A_3$.  To be concrete, let us choose
    the following values of the coefficients, which ensure that this quantity vanishes
    (and also that the polynomials do \emph{not} vanish at the other fixed points in
    the ambient space):
    $(A_0, A_1, A_2, A_3) = (-30, 1, 2, 3)$.
    
    With the aid of a computer algebra package,\footnote{I used Mathematica 8
    \cite{Mathematica} for these calculations, as well as to produce the diagrams in
    this paper.} it can be checked that for the above choice of coefficients, $\Xt$ is
    singular at precisely the two fixed points of $g_{10}$.  Our next task is to
    investigate the nature of these singularities.  To this end, we choose local
    coordinates centred on one of the fixed points (the other is identical,
    thanks to the $g_2$ symmetry exchanging them): set $t_{i,0} = 1$ for all $i$,
    and $t_{i,1} = 1 + \e_i$.  We then expand the polynomials $p_1, p_2$ to
    quadratic order in $\vec\e$:
    \begin{align*}
        p_1 =& 18 \e_0 + 18 \e_1 + 18 \e_2 + 18 \e_3 + 18 \e_4 + 10 \e_0 \e_1 + 11 \e_0 \e_2\\
        &+ 11 \e_0 \e_3 + 10 \e_0 \e_4 + 10 \e_1 \e_2
            + 11 \e_1 \e_3 + 11 \e_1 \e_4 + 10 \e_2 \e_3 + 11 \e_2 \e_4 + 10 \e_3 \e_4 + \cO(\e^3) \\
        p_2 =& -18 \e_0 - 18 \e_1 - 18 \e_2 - 18 \e_3 - 18 \e_4 - 26 \e_0 \e_1 - 25 \e_0 \e_2 \\
        &- 25 \e_0 \e_3 - 26 \e_0 \e_4 - 26 \e_1 \e_2 - 25 \e_1 \e_3
             - 25 \e_1 \e_4- 26 \e_2 \e_3 - 25 \e_2 \e_4 - 26 \e_3 \e_4 + \cO(\e^3)
    \end{align*}
    Note that $p_1 + p_2$ has no linear term, and is invariant under the action of
    $g_{10}$.  To proceed, we can solve (say) $p_2 = 0$ for $\e_0$ as a power
    series in the other variables, and substitute this into $p_1+p_2$ up to order
    $\e^2$.  I will spare the reader the ugly details, but one can check that the
    resulting leading-order terms constitute a non-degenerate quadratic form,
    invariant under the $g_{10}$ action.  So our singularity is analytically
    isomorphic to the conifold, and the quotient space has a
    $\IZ_{10}$-hyperconifold singularity.
    
    To determine the $\IZ_{10}$ action on the singularity, we linearise the action
    on the local coordinates.  In terms of the $\e_i$, \eqref{eq:Z10Z2action}
    becomes
    \begin{equation*}
        g_{10} = [1 : 1 + \e_i] \mapsto [1+\e_{i+1} : 1] \sim \left[1 : \frac{1}{1+\e_{i+1}}\right]
            = [1 : 1 - \e_{i+1} + \cO(\e^2)] ~.
    \end{equation*}
    In other words $\e_i \mapsto -\e_{i+1}$ for $i=1,2,3$.  For $i=4$, we have
    $\e_4 \mapsto -\e_0$, but $\e_0$ is not amongst our local coordinates; we
    have solved $p_2 = 0$ to obtain $\e_0 = -\e_1 - \e_2 - \e_3 - \e_4 + \cO(\e^2)$,
    so the $g_{10}$ action effects $\e_4 \mapsto \e_1 + \e_2 + \e_3 + \e_4$.
    The linearised action of $g_{10}$ is therefore given by the matrix
    \begin{equation*}
        \left(\begin{array}{rrrr}
            0 & -1 & 0 & 0 \\
            0 & 0 & -1 & 0 \\
            0 & 0 & 0 & -1 \\
            1 & 1 & 1 & 1
        \end{array}\right) ~.
    \end{equation*}
    The eigenvalues of this matrix are $\z, \z^3, \z^7,\z^9$, where
    $\z = e^{2\pi\ii/10}$, so the singularity on the quotient space is described by
    the hyperconifold $\cC_{10,3}$.
    
    We conclude that, over the locus $A_0 + 5A_1 + 5A_2 + 5A_3 = 0$, the
    space $X^{1,3}$ develops a hyperconifold singularity, modelled on
    $\cC_{10,3}$.  By \ref{th:factoriality}, it is factorial, and by
    \ref{th:global_resolution} we can therefore obtain a new Calabi--Yau manifold
    by deleting the singular point, and gluing in any geometry described by a star
    subdivision\footnote{In fact, any crepant resolution with a non-empty local
    ample/K\"ahler cone will suffice.  The utility of star subdivisions is that they
    automatically give rise to local ample divisors.} of the fan for $\cC_{10,3}$.  One
    such subdivision is shown in Figure~\ref{fig:C103_resolution}.
    
    We began with a manifold which has fundamental group $\IZ_{10}\times\IZ_2$ and
    Hodge numbers $\hodgenos = (1,3)$, and constructed a
    $\IZ_{10}$-hyperconifold transition by allowing the $\IZ_{10}$ factor to develop
    a fixed point on the covering space.  According to the formulae of
    Theorem \ref{th:hodgenos} and Theorem \ref{th:pi1}, the new manifold has
    Hodge numbers $\hodgenos = (10,2)$, and fundamental
    group $\IZ_2$.
\begin{figure}
\begin{center}
    \includegraphics[width=.3\textwidth]{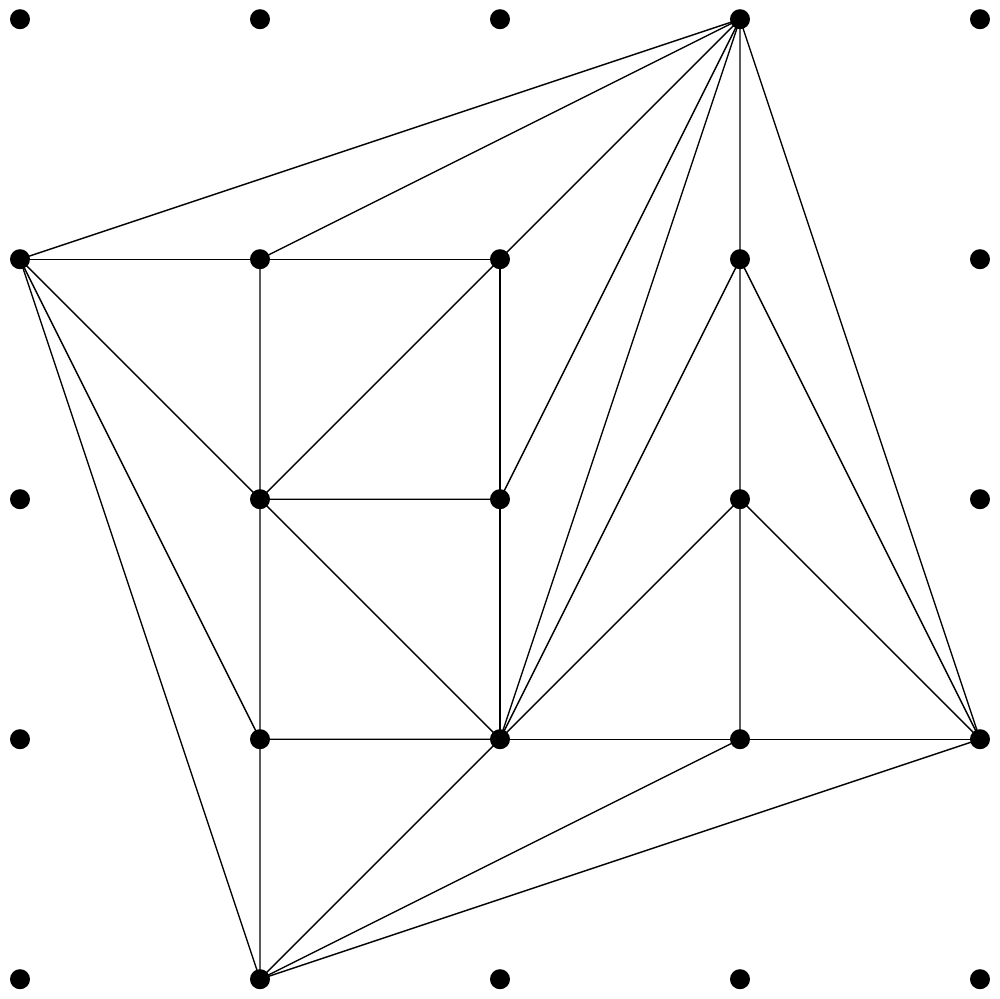}
    \caption{\small \label{fig:C103_resolution}
    One projective crepant resolution of $\cC_{10,3}$, given by a
    sequence of star subdivisions.  A different example, in which the first subdivision
    is with respect to the central lattice point, is given in \cite{Davies:2011is}.
    }
\end{center}
\end{figure}
\end{example}
\begin{example}[A non-normal subgroup]
    In \cite{Braun:2010vc}, Braun found families of smooth Calabi--Yau threefolds
    on which the non-Abelian group $\Dic_3 \cong \IZ_3\rtimes\IZ_4$ acts freely.  There
    is only one non-trivial semi-direct product of this form; the group presentation is
    \begin{equation*}
        \Dic_3 \cong \langle g_3, g_4 ~|~ g_3^3 = g_4^4 = e ~,~ g_4^{-1}g_3 g_4 = g_3^2 \rangle,
    \end{equation*}
    where $e$ is the identity.  It is easy to see that the $\IZ_2$ subgroup generated
    by $g_4^2$ is normal,\footnote{Note that $\IZ_2$ does \emph{not} have a
    complement in $\Dic_3$, so Theorem \ref{th:factoriality} does not apply.
    Nevertheless, blowing up the singular point of a $\IZ_2$-hyperconifold always
    gives a projective crepant resolution.} and that
    $\Dic_3/\IZ_2 \cong S_3$.  This fact was used in \cite{Davies:2011fr} to construct,
    via a $\IZ_2$-hyperconifold transition, the first known Calabi--Yau
    threefold\footnote{In fact, $\Dic_3$ acts freely on three distinct families of smooth
    Calabi--Yau threefolds, all connected by conifold transitions.  Each of the quotients
    can undergo a $\IZ_2$-hyperconifold transition, leading to three spaces with
    fundamental group $S_3$, which in turn are connected by conifold transitions
    \cite{Braun:2009qy}.} with fundamental group $S_3$.
    
    We can instead construct a $\IZ_4$-hyperconifold transition from the space
    $X^{1,4} = \widetilde X^{8,44}/\Dic_3$ \cite{Braun:2009qy}.  This can indeed be
    done; the details can be checked following the same procedure as in the last
    example, and we obtain a new smooth Calabi--Yau $X^{4,3}$.  Because the
    subgroup $\IZ_4$ is not normal, we must compute its normal closure in order
    to find the fundamental group of $X^{4,3}$.
    
    Note that $g_3 g_4 g_3^{-1} = g_4 g_3$, and between them, $g_4$ and
    $g_4 g_3$ generated all of $\Dic_3$.  Therefore $\IZ_4^{\Dic_3} = \Dic_3$;
    Theorem \ref{th:pi1} then implies that $\pi_1(X^{4,3})$ is trivial.
\end{example}
\begin{example}[`Swiss-cheese' manifolds from hyperconifold transitions]

A popular idea in string theory model building is the `LARGE volume scenario'
in Type IIB \cite{Balasubramanian:2005zx,Conlon:2005ki,Cicoli:2008va}.  This
requires a three-dimensional Calabi--Yau manifold admitting a particular limit
in K\"ahler moduli space.  It must be possible to take the overall volume of the
manifold to infinity while certain prime divisors remain of finite size.  Such
spaces have been dubbed `Swiss-cheese manifolds'.

Hyperconifold transitions provide an easy and systematic way to construct such
Swiss-cheese manifolds, which complements the use of the several known
ad-hoc examples, and the computationally-intensive search technique outline
in \cite{Gray:2012jy}.
\end{example}

\newpage
\begin{example}[A non-factorial hyperconifold with a Calabi--Yau resolution]
    In \cite{Balasubramanian:2012wd}, a\linebreak technique was introduced for finding
    embeddings of certain toric singularities into compact\linebreak Calabi--Yau manifolds.
    Briefly, let $Y$ be an affine toric Calabi--Yau.  To embed $Y$ in a\linebreak compact
    Calabi--Yau, the authors of \cite{Balasubramanian:2012wd} seek a compact
    toric Fano fourfold which\linebreak contains an affine patch $Y\times\IC^*$, such that a
    generic anti-canonical hypersurface intersects it in $Y\times\{\pt\}$.
    
    Using the technique described above, the $\IZ_3$-hyperconifold\footnote{In
    \cite{Balasubramanian:2012wd}, this was referred to as (the cone over)
    $Y^{3,0}$, following earlier AdS/CFT literature \cite{Gauntlett:2004yd}.  The
    cones over the spaces $Y^{p,0}$ are exactly the hyperconifolds $\cC_{p,1}$,
    which can be seen by comparing the toric data in \cite{Martelli:2004wu} with
    that from Lemma \ref{lem:diagrams}.} $\cC_{3,1}$ was embedded in a compact
    Calabi--Yau threefold.  The ambient toric space is described by the face fan
    of the four-dimensional polytope with the following vertices:
    \begin{equation*}
        \begin{array}{rrrrrrr}
            v_1 & v_2 & v_3 & v_4 & v_5 & v_6 & v_7 \\
             0 & 0 & 0 & 1 & 0 & -1 & 0 \\
             0 & 0 & \phantom{-}1 & \phantom{-}1 & -1 & 0 & \phantom{-}2 \\
             -1 & 0 & 1 & 2 & 2 & 2 & 2 \\
             0 & -1 & 1 & 3 & 3 & 3 & 3
        \end{array}
    \end{equation*}
    It is clear that the last four vertices lie on the same two-dimensional face, and
    the\linebreak corresponding affine patch is precisely $\cC_{3,1}\times\IC^*$ (to see this,
    it helps to sketch the\linebreak positions of the vertices in the plane).  It follows that the
    divisors which correspond to these vertices are non-Cartier.  Indeed, they are
    simply $D_i\times\IC^*$, where the $D_i$ are the four toric (non-Cartier) divisors
    on $\cC_{3,1}$, and therefore also restrict to non-Cartier divisors on the
    Calabi--Yau hypersurface, demonstrating that it is not factorial.
    
    A smooth Calabi--Yau resolution is also constructed, and the corresponding
    local\linebreak resolution of $\cC_{3,1}$ is the `wrong' one as discussed in Example
    \ref{ex:C31_resolutions}.  There is no contradiction,\linebreak because the singular
    variety is not factorial.  Presumably there is no smoothing of the singular
    variety, so that this example does not correspond to a transition.
\end{example}
\begin{example}[A smoothable non-factorial hyperconifold]
    In \cite{Braun:2011hd}, Calabi--Yau manifolds with Hodge numbers
    $\hodgenos = (1,1)$ were constructed via free actions of groups of order
    twenty-four on a simply-connected manifold with $\hodgenos = (20,20)$.
    The three different groups have a common $\IZ_3$ subgroup (this is not
    maximal; the intersection of the three groups is $\IZ_6$), and all act
    freely on the \emph{same} one-parameter family.
    
    One might be interested in finding $\IZ_3$-hyperconifold transitions from
    these spaces, to\linebreak construct new (rigid) Calabi--Yau manifolds with Hodge
    numbers $\hodgenos = (3,0)$.\linebreak  However, if we arrange for $X^{1,1}$ to
    develop a $\IZ_3$-hyperconifold, it in fact develops four such singularities.
    The reason for this is the extra symmetry; the one-parameter family has a
    larger symmetry group than just the order-24 quotient group, and nodes
    must fall into orbits under the larger symmetry group.
    
    If we specialise the hypersurface to intersect the $\IZ_3$-fixed points,
    then its equation\linebreak simplifies as follows.  Take toric coordinates
    $t_1, t_2, t_3, t_4$ on the four-torus, and define\linebreak coordinates centred on
    one $\IZ_3$-fixed point by letting $t_i = e^{2\pi\ii/3} + x_i$.  Then the
    hypersurface equation becomes
    \begin{equation*}
        x_1 x_4 f_1(\vec x) + x_2 x_3 f_2(\vec x) = 0 ~,
    \end{equation*}
    for certain polynomials $f_1, f_2$.  Clearly, the hypersurface is not factorial
    in this case; one example of a non-Cartier divisor is $x_1 = x_2 = 0$.
    
    Our previous arguments for a Calabi--Yau resolution therefore fail.  However,
    we can certainly obtain resolutions which are complex manifolds with
    $c_1 = 0$, by removing the four $\IZ_3$-hyperconifolds and gluing in
    crepant resolutions of $\cC_{3,1}$.  The resulting spaces have Euler number
    $0 + 4\times 6 = 24$.  I have been unable to establish whether there is a
    projective resolution among these; if so, it would correspond to a rigid
    Calabi--Yau manifold with Hodge numbers $\hodgenos = (12,0)$.  The
    number $h^{1,1} = 12$ must arise from the one Cartier divisor class on
    $X^{1,1}$, the eight `local' divisors which project to the singular points (two to
    each), and the classes of the proper transforms of three independent
    non-Cartier divisors on the singular space.
    
    So this is a more complicated transition than simply a $\IZ_3$-hyperconifold
    transition (or even three such transitions `performed simultaneously'). 
\end{example}
%


\section*{Acknowledgements}

I would like to thank Mark Gross for showing me an example which led directly
to Theorem \ref{th:mirror_sym}, and Karl Schwede for pointing out Lemma
\ref{lem:moving_support}.  I would also like to thank Mark Gross and Bal\'azs
Szendr\H oi for helpful comments about resolutions.  This work was supported
by the Engineering and Physical Sciences Research Council
[grant number EP/H02672X/1].

\appendix

\newpage
\section{The exceptional hyperconifold} \label{app:exception}

In Section~\ref{sec:classification}, we ignored the possibility that the generator
of a hyperconifold quotient group might have a non-trivial projection onto the
$\IZ_2$ factor of the symmetry group.  In this appendix we will show that
there is only one such case---another $\IZ_4$-hyperconifold,
distinct from $\cC_{4,1}$.


%
\begin{lemma} \label{th:exception}
    Up to isomorphism, the only hyperconifold singularity for which the
    quotient group has a non-trivial projection to the $\IZ_2$ factor in
    \eqref{eq:sym_group} is given by the following $\IZ_4$ action:
    \begin{equation*}
        (y_1, y_2, y_3 y_4) \mapsto (\ii y_1, y_3, -y_2, \ii y_4) ~.
    \end{equation*}
\end{lemma}
\begin{proof}
The generator of the $\IZ_2$ factor in the symmetry group of the conifold acts
on the homogeneous coordinates as
$z_1 \leftrightarrow z_3, z_2 \leftrightarrow z_4$, and therefore on the
hypersurface coordinates simply as $y_2 \leftrightarrow y_3$.  Inspecting the
expression \eqref{eq:Omega} for the holomorphic $(3,0)$-form, we see that
it is odd under this transformation: $\O \mapsto -\O$.  Therefore, to obtain a
hyperconifold singularity, we must compose it with some torus
transformation under which $\O$ is also odd.

We have already seen that $\O$ transforms the same way as the polynomial
$p = y_1 y_4 - y_2 y_3$ under the torus action, so we seek finite-order torus
elements under which the monomials $y_1 y_4$ and $y_2 y_3$ are each
odd.  When combined with the order-two transformation above, this leads us
to a subgroup generated by an element of the form
\begin{equation} \label{eq:exceptional_action}
    g : (y_1, y_2, y_3, y_4) \mapsto (\h y_1, \h^k y_3, \h^{n-k} y_2, \h^{n-1} y_4)~,
\end{equation}
where $\h = e^{\pi\ii/n}$ for some $n$, and $0 < k < n$.

Now we must impose that the origin is the only fixed point of any of the group
elements.  First, consider the group action on the linear subspace parametrised
by $(y_2, y_3)$; this is given by the matrix
\begin{equation*}
    \left( \begin{array}{ll}
        0 & \h^{k} \\
        \h^{n-k} & 0 \end{array} \right) ~.
\end{equation*}
We can see immediately that the fourth power of this matrix is the identity,
and therefore that $g^4$ fixes, for example, the whole line
$\{(0, y_2, 0, 0\} \subset \cC$, irrespective of our choice of $n, k$.  Since we
require that any non-trivial group element fixes only the origin, we conclude that
$g$ must have order four. Therefore $\h = \ii$, and $k=0$ or $k=2$.  The
coordinate redefinition $y_2 \leftrightarrow y_3$ exchanges these two
possibilities while leaving the polynomial $p$ invariant, so without loss of
generality, we may assume that $k=0$.
\end{proof}
\begin{remark}
    The polynomial $p$ is odd under the generator of the action described above.
    This singularity therefore does \emph{not} arise from the limit of a free group
    action on the deformed conifold given by $p + t = 0$, $t\neq 0$.
\end{remark}
\newpage
\bibliographystyle{utphys}
\bibliography{references}

\end{document}